\documentclass[twoside,a4paper,12pt]{article}
\usepackage{amssymb, amsmath, amsthm}
\usepackage{color}
\usepackage{enumerate}
\usepackage{hyperref}

\newdimen\myMargin
\textwidth \paperwidth
\advance\textwidth -2\myMargin
\textheight \paperheight
\advance\textheight -2\myMargin
\advance\oddsidemargin \myMargin
\advance\evensidemargin \myMargin
\advance\topmargin \myMargin
\advance\textheight -17 true mm
\advance\topmargin 12.5 true mm

\newcommand{\A}{{\cal A}}

\newcommand{\K}{\mathbb K}

\newcommand{\F}{\mathbb F}

\newcommand{\Z}{\mathbb Z}

\newcommand{\chara}{\mathrm{char}}

\newtheorem{theorem}{Theorem}
\newtheorem{lemma}[theorem]{Lemma}
\newtheorem{corollary}[theorem]{Corollary}
\newtheorem{proposition}[theorem]{Proposition}
\newtheorem{definition}[theorem]{Definition}

\newtheorem{algorithm}{Algorithm}
\newtheorem{fact}[theorem]{Fact}

\theoremstyle{remark}
\theoremstyle{remark}\newtheorem{remark}[theorem]{Remark}

\title{Explicit equivalence of quadratic forms over $\mathbb{F}_q(t)$}

\author{\normalsize
 \begin{minipage}{0.3\linewidth}
    \large
    G\'abor Ivanyos \\
    \footnotesize
Institute for Computer Science and Control,
Hungarian Acad. Sci. \\
    \texttt{Gabor.Ivanyos@sztaki.mta.hu} \\
    \normalsize
  \end{minipage}
  \qquad
 \begin{minipage}{0.3\linewidth}
    \large
    P\'eter Kutas \\
    \footnotesize
Institute for Computer Science and Control,
Hungarian Acad. Sci. and
    Central European University, Department of Mathematics and its Applications \\
    \texttt{Kutas\_Peter@phd.ceu.edu} \\
    \normalsize
  \end{minipage}
  \qquad
  \begin{minipage}{0.3\linewidth}
    \large
    Lajos R\'onyai \\
    \footnotesize
     Institute for Computer Science and Control,
     Hungarian Acad. Sci. \\
    Dept. of Algebra, Budapest
    Univ. of Technology and Economics \\
    \texttt{lajos@ilab.sztaki.hu}
    \normalsize
  \end{minipage}
}

\begin{document}

\maketitle

\begin{abstract}
We propose a randomized polynomial time algorithm for computing non-trivial
zeros of quadratic forms in 4 or more variables over $\mathbb{F}_q(t)$, where
$\mathbb{F}_q$ is a finite field of odd characteristic. The algorithm is based
on a suitable splitting of the form into two forms and finding a common value
they both represent. We make use of an effective formula for the number of
fixed degree irreducible polynomials in a given residue class. We apply our
algorithms for computing a Witt decomposition of a quadratic form, for
computing an explicit isometry between quadratic forms and finding zero
divisors in quaternion algebras over quadratic extensions of $\F_q(t)$. 
\end{abstract}


\bigskip
\noindent
{\textbf{Mathematics Subject Classification:} 68W30, 11E12, 11E20.
\section{Introduction}

In this paper we consider algorithmic questions concerning quadratic forms
over $\mathbb{F}_q(t)$ where $q$ denotes an odd prime power. The main focus is
on the problem of finding a non-trivial zero of a quadratic form. The
complexity of the problem of finding non-trivial zeros of quadratic forms in
three variables has already been considered in (\cite{Cremona1},\cite{IKR}).
However the same problem concerning quadratic forms of higher dimensions
remained open.

Similarly, in the the case of quadratic forms over $\mathbb{Q}$, the
algorithmic problem of finding non-trivial zeros of 3-dimensional forms was
considered in several papers (\cite{CR},\cite{IS}) and afterwards Simon and
Castel proposed an algorithm for finding non-trivial zeros of quadratic forms
of higher dimensions (\cite{Simon},\cite{Castel}). The algorithms for the
low-dimensional cases (dimension 3 and 4) run in polynomial time if one is
allowed to call oracles for integer factorization. Surprisingly, the case
where the quadratic form is of dimension at least 5, Castel's algorithm runs
in polynomial time without the use of oracles. Note that, by the classical
Hasse-Minkowski theorem, a 5-dimensional quadratic form over $\mathbb{Q}$ is
always isotropic if it is indefinite.

Here we consider the question of isotropy of quadratic forms in 4 or more
variables over $\mathbb{F}_q(t)$. The main idea of the algorithm is to split
the form into two forms and find a common value they both represent. Here we
apply two important facts. There is an effective bound on the number of
irreducible polynomials in an arithmetic progression of a given degree. An
asymptotic formula, which is effective for large $q$, was proven by Kornblum
\cite{Kornblum}, but for our purposes, we apply a version with a much better
error term, due to Rhin \cite[Chapter 2, Section 6, Theorem 4]{Rhin}. However,
that statement is slightly more general; hence we cite a specialized version
from \cite{Wan}. A short survey on the history of this result can be found in
\cite[Section 5.3.]{EH}. The other fact we use is the local-global principle
for quadratic forms over $\F_q(t)$ due to Rauter \cite{Rauter}. 

Finally we solve these two equations separately using the algorithm from
\cite{Cremona1} and our Algorithm 1 in the 5-variable case. In the
4-dimensional case we are also able to detect if a quadratic form is
anisotropic; note that a 5-dimensional form over $\mathbb{F}_q(t)$ is always
isotropic. The algorithms are randomized and run in polynomial time. We also
give several applications of these algorithms. Most importantly, we propose an
algorithm which computes a transition matrix of two equivalent quadratic
forms.

The paper is divided into five sections. Section 2 provides theoretical and
algorithmic results concerning quadratic forms over fields. Namely, we give a
general introduction over arbitrary fields and then over $\mathbb{F}_q(t)$,
which is followed by a version of the Gram-Schmidt orthogonalization
procedure which gives control of the size of the output. 

%quadratic forms 

In Section 3 we present the crucial ingredients of our algorithms. In Section
4 we describe the main algorithms and analyze their running time and the size
of their output. In Section 5 we use the main algorithms to compute explicit
equivalence of quadratic forms. In the final section we apply our results to
find zero divisors in quaternion algebras over quadratic extensions of
$\mathbb{F}_q(t)$ or, equivalently, to find zeros of ternary quadratic forms
over quadratic extensions of $\mathbb{F}_q(t)$. The material of this part is
the natural analogue of that presented in \cite{K} over quadratic number
fields.

\section{Preliminaries}
This section is divided into five parts. 
The first recalls the basics of the algebraic 
theory of quadratic forms and quadratic spaces over an arbitrary field
 of characteristic different from 2. In the second part we give a brief overview of
valuations of the field $\mathbb{F}_q(t)$ where $q$ denotes an odd prime power.
The third part is devoted to some results about quadratic forms over 
$\mathbb{F}_q(t)$ that we will use later on. It is followed by a discussion
of a version of the Gram-Schmidt orthogonalization procedure over $\mathbb{F}_q(t)$ with
complexity analysis. The section is concluded
with some known algorithmic results about finding non-trivial zeros of binary and ternary quadratic 
forms over $\mathbb{F}_q(t)$.

\subsection{Quadratic forms over fields}

\label{subsec:qf-gen}

This subsection is based on Chapter I of \cite{Lam}. 
Here $\F$ will denote a field such that $\chara~ \F\neq 2$.

A {\em quadratic form} over $\F$ is a homogeneous polynomial $Q$ of degree two
in $n$ variables  $x_1,\ldots,x_n$ for some $n$. Two quadratic forms
are called {\em equivalent} if they can be obtained from each other by a
homogeneous linear change of the variables. By such a change we mean 
that each variable $x_j$ is substituted by
the polynomial $\sum_{i=1}^nb_{ij}x_i$ ($j=1,\ldots,n$). 
The $n\times n$ matrix $B=(b_{ij})$ over $\F$ has to be invertible 
as otherwise there is no appropriate substitution in the reverse
direction. The {\em matrix} of $Q$ is the unique symmetric
$n$ by $n$ matrix $A=(a_{ij})$ with
$Q(x_1,\ldots,x_n)=\sum_{i=1}^n\sum_{j=1}^n a_{i,j}x_ix_j$. We will also refer to this as the Gram matrix of the quadratic form. The {\em determinant} of a quadratic form is the determinant of its matrix. 
We call $Q$ {\em regular} if its matrix has non-zero determinant
and {\em diagonal} if its matrix is diagonal. We say that
$Q$ is {\em isotropic} if the equation $Q(x_1,\ldots,x_n)=0$ admits
a non-trivial solution and {\em anisotropic} otherwise. Two quadratic forms 
with Gram matrices $A_1$ and $A_2$ are then {\em equivalent} if 
and only if there exists an invertible $n$ by $n$ matrix $B\in M_n(\F)$, 
such that $A_2=B^TA_1B$; equivalently,
$A_1={B^{-1}}^TA_2B^{-1}$. Here $B$ is just the matrix of the change of
variables defined above. We will use the term {\em transition matrix}
for such a $B$. Two regular unary quadratic forms $ax^2$ and $bx^2$
are equivalent if and only if $a/b$ is a square in $\F^*$. In other 
words, equivalence classes of regular unary quadratic forms correspond
to the elements of the factor group $\F^*/(\F^*)^2$.

Every quadratic form is equivalent to a diagonal one, see the
discussion of Gram-Schmidt orthogonalization in the context
of quadratic spaces below and in 
Subsection~\ref{subsec:Gram-Schmidt}. A regular
diagonal quadratic form $Q(x_1,x_2)=a_1x_1^2+a_2x_2^2$ is isotropic
if and only if $-a_2/a_1$ is a square in $\F^*$. Binary quadratic
forms that are regular and isotropic at the same time are called
{\em hyperbolic}. If $(\beta_1,\beta_2)$ is a non-trivial
zero of $Q$ then $\gamma=2(a_1\beta_1^2-a_2\beta_2^2)$ is non-zero
and the substitution $x_1\leftarrow \beta_1x_1+\frac{\beta_1}{\gamma}x_2$,
$x_2\leftarrow \beta_2x_1-\frac{\beta_2}{\gamma}x_2$ provides an equivalence
of $Q$ with the form $x_1x_2$. Another, diagonal standard hyperbolic
is $x_1^2-x_2^2$. The standard forms $x_1x_2$ and $x_1^2-x_2^2$
are equivalent via the substitution  
$x_1\leftarrow \frac{1}{2}x_1 + \frac{1}{2}x_2$,
$x_2\leftarrow \frac{1}{2}x_1 - \frac{1}{2}x_2$; the inverse of this substitution is $x_1\leftarrow x_1+x_2$,
$x_2\leftarrow x_1-x_2$.

A regular ternary quadratic form is equivalent to a diagonal
form $c(ax_1^2+bx_2^2-abx_3^2)$ for some $a,b,c\in \F^*$.
To see this, note that the diagonal form 
$a_1x_1^2+a_2x_2^2+
a_3x_3^2$ is equivalent to 
$$
-a_1a_2a_3(
\frac{-a_1}{a_1a_2a_3}x_1^2+
\frac{-a_2}{a_1a_2a_3}x_2^2-
\frac{a_1a_2}{(a_1a_2a_3)^2}x_3^2)=
a_1x_1^2+a_2x_2^2+\frac{1}{a_3}x_3^2
$$
via the substitution $x_3\rightarrow \frac{1}{a_3}x_3$.
A related object is the {\em quaternion algebra} $H_\F(a,b)$ with $a,b\in\mathbb{F}^*$.
This is the associative algebra over $\F$ with identity element, generated by
$u$ and $v$ with defining relations $u^2=a$, $v^2=b$, $uv=-vu$. 
It can be readily seen that $H_\F(a,b)$ is a four-dimensional
algebra over $\F$ with basis $1,u,v,uv$ whose center is the subalgebra
consisting of the multiples of $1$. It is also known that 
$H_\F(a,b)$ is either a division algebra or it is isomorphic
to the full $2$ by $2$ matrix algebra over $\F$. Any non-zero element $z$
of $H_\F(a,b)$ with $z^2=0$ can be written as a linear
combination of $u,v$ and $uv$. Also, 
$(\alpha_1u+\alpha_2v+\alpha_3uv)^2=(a\alpha_1^2+b\alpha_2^2-
ab\alpha_3^2)1$, where $\alpha_1,\alpha_2,\alpha_3\in\F$. Hence finding a 
non-zero nilpotent element $z$ of
$H_\F(a,b)$ is equivalent to computing a non-trivial
zero of the quadratic form $ax_1^2+bx_2^2-abx_3^2$.
In particular, isotropy of $ax_1^2+bx_2^2-abx_3^2$
is equivalent to $H_\F(a,b)$ being isomorphic to a full
matrix algebra. 

It will be convenient to present certain parts of this paper in the framework
of {\em quadratic spaces}. These offer a coordinate-free approach to quadratic
forms. A quadratic space over $\F$ is a pair $(V,h)$ consisting of a vector
space $V$ over $\F$ and a symmetric bilinear function $h:V\times V\rightarrow
\F$. Throughout this paper all vector spaces will be finite dimensional. To a
quadratic form $Q$ having Gram matrix $A$ the associated bilinear function $h$ is 
$h(u,v)=u^TAv$ for $u,v\in\F^n$. Conversely, if $(V,h)$ is an $n$-dimensional
quadratic space, then, for any basis $v_1,\ldots,v_n$, we can define its {\it
Gram matrix} $A=(a_{ij})$ with respect to the given basis by putting
$a_{ij}=h(v_i,v_j)$. Then $Q(x_1,\ldots,x_n)={\underline x}^TA{\underline x}$
is a quadratic form where $\underline x$ stands for the column vector
$(x_1,\ldots,x_n)^T$. The quadratic form obtained from $h$ using
another basis will be a form equivalent to $Q$. Let $(V,h)$ and $(V',h')$ be
quadratic spaces. Then a linear bijection $\phi:V\rightarrow V'$ is an {\em
isometry} if $h'(\phi(v_1),\phi(v_2))=h(v_1,v_2)$ for every $v_1,v_2\in V$. We
say that $(V,h)$ and $(V',h')$ are {\it isometric} if there is an isometry
$\phi:V\rightarrow V'$. Equivalent quadratic forms give isometric quadratic
spaces and to isometric quadratic spaces equivalent quadratic forms are
associated. Moreover, the following holds. Let $(V,h)$ and $(V',h')$ be
quadratic spaces. Let $v_1,\dots,v_n$ be a basis of $V$ and let
$v_1',\dots,v_n'$ be a basis of $V'$. Suppose that $\phi$ is an isometry
between $V$ and $V'$. Then $\phi(v_i)=\sum_{j=1}^n b_{ij}v_j'$ where
$b_{ij}\in\F$. Let $A$ be the Gram matrix of $h$ in the basis $v_1,\dots,v_n$
and let $A'$ be the Gram matrix of $h'$ in the basis $v_1',\dots,v_n'$. If
$B\in M_n(\F)$ is equal to the matrix $(b_{ij})$ then $A=B^TA'B$. 

%To be more specific, let $A_1$ and $A_2$ be the Gram matrices of $h_1$ and
%$h_2$ in a
%basis of $V_1$ and a basis of ~$V_2$, respectively. Let $\phi:V_1\rightarrow
%V_2$ be an isometry with matrix $B$ with respect to the given pair of bases.
%Then we have $A_2=B^TA_1B$.

Let $(V,h)$ be a quadratic space. We say that two vectors $u$ and $v$ from $V$
are {\em orthogonal} if $h(u,v)=0$. An orthogonal basis is a basis consisting
of pairwise orthogonal vectors. The well-known Gram-Schmidt orthogonalization
procedure provides an algorithm for constructing orthogonal bases. We will
discuss some details in the context of quadratic spaces over $\F_q(t)$ in
Subsection~\ref{subsec:Gram-Schmidt}. With respect to an orthogonal basis, the
Gram matrix is diagonal. Therefore the Gram-Schmidt procedure gives a way of
computing diagonal forms equivalent to given quadratic forms. The {\em
orthogonal complement} of a subspace $U\leq V$ is the subspace
$$U^\perp=\{v:h(u,v)=0\mbox{~for every~}u\in U\}.$$ The subspace $V^\perp$ is
called the {\em radical} of $(V,h)$; here, $(V,h)$ is called regular if its radical
is zero. A quadratic space is regular if and only if at least one of, or
equivalently, each of the quadratic forms associated to it using various bases
is regular. 

The {\em orthogonal sum} of $(V,h)$ and $(V',h')$ is the quadratic space
$(V\oplus V', h\oplus h')$ where $h\oplus
h'((v_1,v_1'),(v_2,v_2'))=h(v_1,v_2)+h'(v_1',v_2')$; here, $v_1,v_2\in V$ and
$v_1',v_2'\in V'$. The inner version of this is a decomposition of $V$ into
the direct sum of two subspaces $V$ and $V'$ with $V\leq V'^\perp$ and $V'\leq
V^\perp$. An orthogonal basis gives a decomposition into the orthogonal sum of
one-dimensional quadratic spaces.

A non-zero vector in a quadratic space is called {\em isotropic} if it is
orthogonal to itself. Isotropic vectors correspond to non-trivial zeros of
quadratic forms. A quadratic space is isotropic if it admits isotropic vectors
and {\em anisotropic} otherwise. A quadratic space $(V,h)$ is {\em totally
isotropic} if $h$ is identically zero on $V\times V$. This is equivalent to
that every non-zero vector in $V$ is isotropic; here, $\chara~\F\neq 2$.
Every subspace $U\leq V$ in a quadratic space $(V,h)$ is also a quadratic
space with the restriction of $h$ to $U$. A subspace of $V$ is called
isotropic, anisotropic, totally isotropy if it is isotropic,
anisotropic, totally isotropic as a quadratic space with the restriction of
$h$. A quadratic space can be decomposed as an orthogonal sum of a totally
isotropic subspace, necessarily the radical of the whole space, and a
regular space, which can actually be any of the direct complements of the
radical. A two-dimensional quadratic space is called a {\em hyperbolic plane}
if it is regular and isotropic. Such spaces correspond to hyperbolic binary
forms.

\begin{theorem}[Witt]\label{Witt}
Let $(V,h)$ be a quadratic space over $\F$. Then $V$ can be decomposed as the
orthogonal sum of $V_0$, a totally isotropic space, $V_h$, which is an
orthogonal sum of hyperbolic planes, and an anisotropic space $V_a$. Such a
decomposition is called a Witt decomposition of $(V,h)$ and the number
$\frac{1}{2}\dim(V_h)$ is called the Witt index of $(V,h)$. Here $V_0$ is the
radical. The Witt index and the isometry class of the anisotropic part $V_a$
do not depend on the particular Witt decomposition. In turn, two quadratic
spaces are isometric if and only if their radicals have the same dimension,
their Witt indices coincide and their anisotropic parts are isometric.
\end{theorem}

A proof of this theorem can be found in \cite[Chapter I, Theorem 4.1.]{Lam}.
There is another interpretation of the Witt index concerning totally isotropic
subspaces.

\begin{proposition}
Let $(V,h)$ be a regular quadratic space with Witt index $m$. 
Then the dimension of every maximal totally isotropic subspace is $m$.
\end{proposition}

The proof of this proposition can be found in \cite[Chapter I, Corollary
4.4.]{Lam}. By the following fact, the Witt decomposition has implications to
equivalence of quadratic forms. 

\begin{proposition}
Two regular quadratic spaces $(V,h)$ and $(V',h')$ 
having the same dimension are isometric
if and only if the orthogonal sum of $(V,h)$ and $(V',-h')$
can be decomposed as an orthogonal sum of hyperbolic planes.
\end{proposition}

The proof of this proposition can be found in \cite[Proposition
2.46.]{Gerstein}.

Thus, deciding isotropy of quadratic spaces or, equivalently, deciding
equivalence of quadratic forms can be reduced to
computing Witt decompostions. In Chapter 5 we will show that
such a reduction exists even for computing isometries 
and explicitly for computing transition matrices.

\subsection{Valuations and completions of $\mathbb{F}_q(t)$}

\label{subsec:val-fct}

We recall some facts about valuations (\cite{Neukirch}). A {\em discrete} (exponential) {\em valuation} of a field $K$
is a map $v:K\rightarrow \Z\cup\{\infty\}$ such that for every $a,b\in K$,
(1) $v(a)=\infty$ if and only
if $a=0$, (2) $v(ab)=v(a)+v(b)$ and (3) $v(a+b)\geq \min\{v(a),v(b)\}$.
A valuation is called trivial if $v(a)$ is identically zero on $K\setminus\{0\}$.
Let $v$ be a non-trivial discrete valuation of $K$ and let $r$ be any  
real number greater than one. Then $d_{v,r}(a,b)=r^{-v(a-b)}$ 
is a metric on $K$. The topology induced on $K$ by this metric does not 
depend on the choice of $r$ and will also remain the same if we replace 
$v$ with a multiple by any positive integer. Let $K_v$ be the completion of $K$ 
with respect to any of the metrics $d_{v,r}$. The natural extension
of the field operations $K_v$ makes $K_v$ a field. Furthermore, 
a natural extension of $v$ is a discrete valuation of $K_v$. 
The elements $a$ of $K$ with $v(a)\geq 0$  form a subring of $K$,
called the {\em valuation ring} corresponding to $v$. The valuation ring
is a local ring in which every ideal is a power of the maximal ideal,
called the {\em valuation ideal}, consisting
 of the elements $a$ with $v(a)>0$. The {\em residue field}
is the factor of the valuation ring by the valuation ideal.

We define the degree of a non-zero rational function from $\mathbb{F}_q(t)$ as
the difference of the degrees of its numerator and denominator. Together with
the convention that the degree of the zero polynomial is $-\infty$, the
negative of the degree function, that is, the degree of the denominator minus
the degree of the numerator, gives a discrete valuation of $\mathbb{F}_q(t)$.
This is the {\em valuation at infinity}. All the other non-trivial valuations
are associated to irreducible polynomials from $\mathbb{F}_q[t]$ via the
following construction (\cite[Theorem 3.15.]{Gerstein}). If $f(t)\in
\mathbb{F}_q[t]$ is an irreducible polynomial, then we can define $v_f(h(t))$
as the difference of the multiplicities of $f(t)$ in the denominator and
numerator of $h(t)$. We will denote by $\mathbb{F}_q(t)_{(f)}$ the completion
of $\mathbb{F}_q(t)$ with respect to $v_f$. As an example, for $f(t)=t$,
$\mathbb{F}_q(t)_{(t)}$ is isomorphic to the field of Laurent series in $t$
over $\mathbb{F}_q$ and the valuation ring inside this consists of the power
series in $t$. We remark that the valuation at infinity can be obtained in a
similar way: Put $t'=1/t$. Then every non-zero polynomial $g(t)\in
\mathbb{F}_q[t]$ can be written as ${t'}^{-\deg g(t)}$ times a polynomial from
$\mathbb{F}_q[t']$ with non-zero constant term. It follows that the degree of
a rational function in $t$ coincides with the difference of the exponents of
the highest power of $t'$ dividing a pair polynomials in $t'$ expressing the
same function as a fraction. This implies that the completion of $\F_q(t)$
with respect to the negative of the degree function is $\F_q(\frac{1}{t})$,
the field of formal Laurent series in $\frac{1}{t}$. 

We  refer to equivalence classes of valuations as {\em primes} of 
$\mathbb{F}_q(t)$. The term {\em infinite prime} or {\em infinity} is
used for valuations equivalent to the negative of the degree, while the {\em finite} primes
correspond to the monic irreducible polynomials of $\mathbb{F}_q[t]$. We shall refer to certain properties satisfied at the
completion corresponding to a prime, for example, isotropy of a quadratic form
over $\mathbb{F}_q(t)$,
 as behaviors {\em at the prime}.

\subsection{Quadratic forms over $\mathbb{F}_q(t)$}

\label{subsec:qf-fct}

In this subsection we recall some basic facts about quadratic forms over
$\mathbb{F}_q(t)$ and its completions, where $q$ is an odd prime power.
The main focus is on the question of isotropy of such forms. We start with two
useful facts concerning quadratic forms over finite fields. The first
one was already established earlier in Section 2.1.
\begin{fact}\label{fact1}
\begin{itemize}

\item[\rm(1)] Let $a_1x_1^2+a_2x_2^2$ be a non-degenerate quadratic form over
a field $\mathbb{F}$. Then it is isotropic if and only if $-a_1a_2$ is a
square in $\mathbb{F}$.

\item[\rm(2)] Every non-degenerate quadratic form over $\mathbb{F}_q$ with at
least three variables is isotropic.
\end{itemize}
\end{fact}

\begin{remark}\label{square2}
If $\mathbb{F}=\mathbb{F}_q$ then to check whether an element $s\neq 0$ in
$\mathbb{F}$ is a square or not, compute $s^{\frac{q-1}{2}}$ and check
whether it is 1 or -1. Hence due to Fact \ref{fact1} there is a deterministic
polynomial time algorithm for checking whether $a_1x_1^2+a_2x_2^2=0$ is
solvable over $\mathbb{F}_q$ or not.
\end{remark}

Now we turn our attention to quadratic forms over $\mathbb{F}_q(t)$ and their
completions. The first lemma deals with quadratic forms in three variables.

\begin{lemma}\label{Lemma1}
 Let $a_1,a_2,a_3\in\mathbb{F}_q[t]$ be non-zero polynomials. Let $f$ be a
 monic irreducible polynomial. Let $\mathbb{F}_q(t)_{(f)}$ denote the $f$-adic
 completion of $\mathbb{F}_q(t)$. Let $v_f(a_i)$ denote the multiplicity of
 $f$ in the prime decomposition of $a_i$. 
\begin{enumerate}

\item[\rm(1)] If $v_f(a_1)\equiv v_f(a_2)\equiv v_f(a_3) ~(mod~ 2)$ then the equation
$a_1x_1^2+a_2x_2^2+a_3x_3^2=0$ is solvable in $\mathbb{F}_q(t)_{(f)}$.

\item[\rm(2)] Suppose that not all the $v_f(a_i)$ have the same parity, and 
 that
$v_f(a_i)\equiv v_f(a_j) ~(mod~ 2)$. Then the equation
$a_1x_1^2+a_2x_2^2+a_3x_3^2=0$ is solvable in $\mathbb{F}_q(t)_{(f)}$ if and
only if $-f^{-v_f(a_ia_j)}a_ia_j$ is a square modulo $f$.

\end{enumerate}
\end{lemma}

\begin{proof}
First assume that all $v_f(a_i)$ have the same parity. By a change of
variables, we may assume
that either $v_f(a_i)=0$ for all $i$ or $v_f(a_i)=1$. In the second case we
can divide the equation by $f$ so we may assume that none of the $a_i$ are
divisible by $f$. We obtain an equivalent form whose coefficients are units in
$\mathbb{F}_q(t)_{(f)}$. An equation $a_1x_1^2+a_2x_2^2+a_3x_3^2=0$ where all
$a_i$ are units in $\mathbb{F}_q(t)_{(f)}$ is solvable by \cite[Chapter VI,
Corollary 2.5.]{Lam}.

Now we turn to the second claim. By a change of variables we may assume that
all the $a_i$ are square-free. This results in two cases. Either $f$ divides
exactly one of the $a_i$ or $f$ divides exactly two of the $a_i$. First we
consider the case where $f$ divides exactly one, say $a_1$ (hence now
$v_f(a_2)=v_f(a_3)=0$ and $v_f(a_1)=1$).

The necessity of $-a_2a_3$ being a square modulo $f$ is trivial since
otherwise the equation $a_1x_1^2+a_2x_2^2+a_3x_3^2=0$ is not solvable modulo
$f$. Now
assume that $-a_2a_3$ is a square modulo $f$. This implies that
$-\frac{a_2}{a_3}$ is a square as well. Note that $-\frac{a_2}{a_3}$ is a unit
in $\mathbb{F}_q(t)_{(f)}$. Hence, by Hensel's lemma, $-\frac{a_2}{a_3}$ is a
square in $\mathbb{F}_q(t)_{(f)}$. Now solvability follows
from Fact \ref{fact1}.

Now let us consider the case where $f$ divides exactly two coefficients, say
$a_2$ and $a_3$. We apply the following change of variables: $x_2\leftarrow
x_2/f$ and $x_3\leftarrow x_3/f$. Now we have the equivalent equation
$a_1x_1^2+a_2(x_2/f)^2+a_3(x_3/f)^2=0$. We multiply this equation by $f$ and
get the equation $fa_1x_1^2+a_2/fx_2^2+a_3/fx_3^2=0$. This equation is
solvable in $\mathbb{F}_q(t)_{(f)}$ if and only if $\frac{-a_2a_3}{f^2}$ is a
square modulo $f$ by the previous point, since $f$ only divides one of the
coefficients.
\end{proof}

The previous lemma characterized solvability at a finite prime. The next one
considers the question of solvability at infinity.
\begin{lemma}\label{Lemma1'}
Let $a_1,a_2,a_3\in\mathbb{F}_q[t]$ be non-zero polynomials. Then the following hold:
\begin{enumerate}[\rm(1)]
\item If the degrees of the $a_i$ all have the same parity then the equation $a_1x_1^2+a_2x_2^2+a_3x_3^2=0$ admits a non-trivial solution in $\mathbb{F}_q((\frac{1}{t}))$.
\item Assume that not all of the degrees of the $a_i$ have the same parity. Also assume that $deg(a_i)\equiv deg(a_j)~(mod~ 2)$. Let $c_i$ and $c_j$ be the leading coefficients of $a_i$ and $a_j$ respectively. Then the equation $a_1x_1^2+a_2x_2^2+a_3x_3^2=0$ has a non-trivial solution in $\mathbb{F}_q((\frac{1}{t}))$ if and only if $-c_ic_j$ is a square in $\mathbb{F}_q$.
\end{enumerate}
\end{lemma}
\begin{proof}
Let $u=1/t$ and $d_i=deg(a_i)$. Substitute $x_i\leftarrow y_iu^{d_i}$. The coefficient of $y_i^2$ becomes $a_i'=u^{2d_i}a_i$. Notice that $a_i'=u^{d_i}b_i$ where $b_i$ is a polynomial in $u$ with non-zero constant term $c_i$. It follows that $v_{u}(a_i')=d_i$ and the residue of $u^{-d_i}a_i$ modulo $u$ is $c_i$. Thus both statements follow from Lemma \ref{Lemma1} applied to $f=u$ in $\mathbb{F}_q[u]$.
%First let us assume that all the $a_i$ have the same parity. Then by suitable change of variables (multiplying $x_i$ with a suitable power of $t$) we may assume that all of them have the same degree $d$ (the new coefficients will be rational functions not necessarily polynomials). Then if we divide the equation $a_1x_1^2+a_2x_2^2+a_3x_3^2=0$ by $t^d$ we get an equation where every coefficient is a unit in $\mathbb{F}_q[[1/t]]$. This is solvable by \cite[Chapter VI, Corollary 2.5.]{Lam}.

%Now we turn to the second part. Let $u:=1/t$. Observe that $\mathbb{F}_q[u]\cong \mathbb{F}_q[t]$. Apply Lemma \ref{Lemma1} for the monic irreducible $f=u$ in $\mathbb{F}_q[u]$.
\end{proof}
\begin{remark}\label{iso}
A four-dimensional form is always isotropic at infinity if three of its coefficient have the same degree parity. Indeed, let $a_i$ be the coefficient whose degree parity is different. Then setting $x_i=0$ and applying Lemma \ref{Lemma1'}, (1) implies the desired result.  
\end{remark}

The next lemmas deal with local solvability of quadratic forms in 4 variables.

\begin{lemma}\label{Lemma1''}
Let $a_1,a_2,a_3,a_4\in\mathbb{F}_q[t]$ be square-free polynomials. Let
$f\in\mathbb{F}_q[t]$ be a monic irreducible dividing exactly two of the
coefficients, $a_i$ and $a_j$. Let the other two coefficients be $a_k$ and
$a_l$. Then the equation $a_1x_1^2+a_2x_2^2+a_3x_3^2+a_4x_4^2=0$ is solvable
in $\mathbb{F}_q(t)_{(f)}$ if and only if at least one of the two conditions
holds:
\begin{enumerate}[\rm(1)]
\item $-a_ka_l$ is a square modulo $f$;
\item $-(a_i/f)(a_j/f)$ is a square modulo $f$.
\end{enumerate}
\end{lemma}

\begin{proof}
First we prove that if any of these conditions hold, the equation is locally
solvable at $f$. If the first condition holds we apply Lemma \ref{Lemma1} to
show the existence of a non-trivial solution with $x_i=0$. If the second
condition holds we apply the following change of variables: $x_i\leftarrow
x_i/f, x_j\leftarrow x_j/f$. With these variables we have the following
equation: 
$$ 
a_i(x_i/f)^2+a_j(x_j/f)^2+a_kx_k^2+a_lx_l^2=0. 
$$ 

By multiplying this equation by $f$ we get an equation where the coefficients
of $x_i$ and $x_j$ are not divisible by $f$ and the the other two are. Now
applying Lemma \ref{Lemma1} again proves the result.

Now we prove the converse. If the equation
$a_1x_1^2+a_2x_2^2+a_3x_3^2+a_4x_4^2=0$ has a solution in
$\mathbb{F}_q(t)_{(f)}$ then it has a solution in the valuation ring of
$\mathbb{F}_q(t)_{(f)}$. We denote this ring by $O$. Let $u_1,u_2,u_3,u_4\in
O$ be a solution satisfying that not all of them are divisible by $f$. Let us
consider the equation modulo $f$:
\begin{equation}
a_1u_1^2+a_2u_2^2+a_3u_3^2+a_4u_4^2\equiv 0~ (mod~ f).
\end{equation}

The rest of the proof is divided into subcases depending on how many of
$u_1,u_2,u_3,u_4$ are divisible by $f$.

If none are divisible by $f$ then we get that $a_ku_k^2+a_lu_l^2\equiv 0 ~(mod
~f)$. Therefore $-a_ka_l$ is a square modulo $f$.

Assume that $f$ divides exactly one of the $u_r$. If $r=i$ or $r=j$ we again
have that $a_ku_k^2+a_lu_l^2\equiv 0 ~(mod~ f)$, so $-a_ka_l$ is again a
square modulo $f$. Observe that $r$ cannot be $k$ or $l$ since then equation
(1) would not be satisfied.

Now consider the case where $f$ divides exactly two of the $u_r$. If $f$
divides $u_i$ and $u_j$ we have again that $a_ku_k^2+a_lu_l^2\equiv 0 ~(mod~
f)$. The next subcase is when $f$ divides exactly one of $u_i$ and $u_j$, and
exactly one of $u_k$ and $u_l$. Assume that $u_i$ and $u_k$ are the ones
divisible by $f$. This cannot happen since then
$a_iu_i^2+a_ju_j^2+a_ku_k^2+a_lu_l^2\equiv a_lu_l^2 ~(mod ~f)$ and hence the
left-hand side of equation (1) would not be divisible by $f$. Finally assume
that $u_k$ and $u_l$ are divisible by $f$. Let $u_k':=u_k/f$ and
$u_l':=u_l/f$. We have that $a_1u_1^2+a_2u_2^2+a_3u_3^2+a_4u_4^2=0$. We divide
this equation by $f$ and obtain the equation
$(a_i/f)u_i^2+(a_j/f)u_j^2+fa_ku_k'^2+fa_lu_l'^2=0$. We have already seen that
this implies that $-(a_i/f)(a_j/f)$ is a square modulo $f$.

Now suppose that three of the $u_r$ are divisible by $f$. Observe that $u_k$
and $u_l$ must be divisible by $f$ since otherwise (1) would not be satisfied.
Assume that $u_i$ is not divisible by $f$. However, this cannot happen,
because $a_1u_1^2+a_2u_2^2+a_3u_3^2+a_4u_4^2 \equiv a_iu_i^2 \not\equiv 0 ~
(mod ~ f^2 )$. 
\end{proof}

 %We have that $a_iu_i^2+a_ju_j^2+a_ku_k^2+a_lu_l^2=0$. Divide this equation by $f$. Then we obtain that $(a_i/f)u_i^2+(a_j/f)u_j^2+fa_k(u_k/f)^2+fa_l(u_l/f)^2=0$. Let $u_k'=u_k/f$ and let $u_l'=u_l/f$. With this notation $(a_i/f)u_i^2+(a_j/f)u_j^2+fa_ku_k'^2+fa_lu_l'^2=0$. Here $f$ does not divide $a_i/f$ and $a_j/f$. Since $f$ does not divide $u_i$, one of $u_j$,$u_k'$ and $u_l'$ must also not be divisible by $f$. So we have that among $u_i,u_j,u_k'$ and $u_l$ at most 2 are divisible by $f$. Applying the results of the previous paragraphs in this proof gives the desired result.  

%The case where $f$ divides exactly three of the $u_r$ is similar to the case where $f$ divides exactly one.

%If $f$ only divides at most 1 one of the $x_i$ then the above equation cannot be true if one of the above conditions does not hold. Indeed, since modulo $f$ it either reduces to 1 non-zero term, or a quadratic form in 2 variables which is not solvable (due to Fact \ref{fact1}). If $f$ divides 3 of $x_i$ then the equation $a_1x_1^2+a_2x_2^2+a_3x_3^2+a_4x_4^2=0$ will not be true either modulo $f$ (if $x_k$ or $x_l$ is the one not divisible by $f$) or modulo $f^2$ (if $x_i$ or $x_j$ is the one not divisible by $f$). Hence exactly two of the $x_i$ are divisible by $f$. If it is $x_i$ and $x_j$ we get the first condition. If it is $x_k$ and $x_l$ we get the second condition. Any other situation cannot occur since then $a_1x_1^2+a_2x_2^2+a_3x_3^2+a_4x_4^2$ would not be congruent to 0 modulo $f$.

The next lemma is the version of Lemma \ref{Lemma1''} at infinity.
\begin{lemma}\label{Lemma1'''}
Let $a_1,a_2,a_3,a_4\in\mathbb{F}_q[t]$ be square-free polynomials. Assume that $a_i$ and $a_j$ are of even degree and the other two, $a_k$ and $a_l$ are of odd degree. Let $c_m$ be the leading coefficient of $a_m$ for $m=1\dots 4$. Then the quadratic form $a_1x_1^2+a_2x_2^2+a_3x_3^2+a_4x_4^2$ is anisotropic in $\mathbb{F}_q((\frac{1}{t}))$ if and only if
$-c_ic_j$ and $-c_kc_l$ are both non-squares in $\mathbb{F}_q$.
\end{lemma}

\begin{proof}

Let $u=1/t$.  By the
substitution $x_r\leftarrow x_rt^{\lceil{\frac{-deg(a_r)}{2}}\rceil}$
 for $r=1,2,3,4$, we obtain new coefficients
$a_r'\in\mathbb{F}_q[u]$. Observe that the $u$ does not divide $a_i'$ and
$a_j'$ and the multiplicity of $u$ in $a_k'$ and $a_l'$ is 1. The remainder of
$a_i'$ modulo $u$ is $c_i$, the remainder of $a_j'$ modulo $u$ is $c_j$. The
remainder of $a_k'/u$ modulo $u$ is $c_k$ and the remainder $a_l'/u$ modulo
$u$ is $c_l$. Hence we may apply Lemma \ref{Lemma1''} with $f=u$ in
$\mathbb{F}_q[u]$. 

%Let us consider the following equation:
%$$a_1x_1^2+a_2x_2^2+a_3x_3^2+a_4x_4^2=0$$

%First we do a change of variables such that those coefficients whose degrees have the same parity become equal (we change $x_i$ to $x_it^r$ for a suitable integer $r$). Then we divide the equation by the smallest power of $t$ after which each coefficient lies in $\mathbb{F}_q[[1/t]]$. If the equation is solvable in $\mathbb{F}_q((1/t))$ then it is also solvable in $\mathbb{F}_q[[1/t]]$. Now exactly two coefficients are divisible by $1/t$. So Lemma \ref{Lemma1''} applies.
\end{proof}
\begin{remark}\label{inf}
If $q\equiv 1 ~(mod ~ 4)$ then the lemma says that anisotropy occurs if and only if exactly one of $c_i$ and $c_j$ is a square and the same holds for $c_k$ and $c_l$. If $q\equiv 3~ (mod ~ 4)$ then the lemma says that anisotropy occurs if and only $c_i$ and $c_j$ are either both squares or both non-squares and the same holds for $c_k$ and $c_l$. The reason for this is that $-1$ is a square in $\mathbb{F}_q$ if and only if $q\equiv 1 ~ (mod ~ 4)$. 
\end{remark}

There is  also  the following fact \cite[Chapter VI, Theorem 2.2]{Lam}.
\begin{fact}\label{fact2}
Let $K$ be 
%a non-dyadic ????? local field, i.e. 
a complete field with respect 
to a discrete valuation whose residue field is a finite field with 
odd characteristic. Then every non-degenerate quadratic form over 
$K$ in 5 variables is isotropic.
\end{fact}

We state a variant of the Hasse-Minkowski theorem over $\mathbb{F}_q(t)~
$\cite[Chapter VI, 3.1]{Lam}. It was proved by Hasse's doctoral student
Herbert Rauter in 1926 \cite{Rauter}.
\begin{theorem}\label{HM}
A non-degenerate quadratic form over $\mathbb{F}_q(t)$ is isotropic over $\mathbb{F}_q(t)$ if and only if it is isotropic over every completion of $\mathbb{F}_q(t)$.
\end{theorem}

For ternary quadratic forms there exists a slightly stronger version of this theorem which is a consequence of the product formula for quaternion algebras or Hilbert's reciprocity law \cite[Chapter IX, Theorem 4.6]{Lam}:
\begin{fact}\label{Product formula}
Let $a_1x_1^2+a_2x_2^2+a_3x_3^2$ be a non-degenerate quadratic form over $\mathbb{F}_q(t)$. Then if it is isotropic in every completion except maybe one then it is isotropic over $\mathbb{F}_q(t)$.
\end{fact}

There is a useful fact about local isotropy of a quadratic form \cite[Chapter VI, Corollary 2.5]{Lam}:
\begin{fact}\label{not divide}
Let $Q(x_1,\dots,x_n)=a_1x_1^2+\dots +a_nx_n^2$ ($n\geq 3$) be a non-degenerate quadratic form over $\mathbb{F}_q(t)$ where $a_i\in\mathbb{F}_q[t]$. If $f\in\mathbb{F}_q[t]$ is a monic irreducible not dividing $a_1\dots a_n$ then $Q$ is isotropic in the $f$-adic completion.
\end{fact}

We finish the subsection with a formula on the number of monic irreducible polynomials of given degree in a residue class (\cite[Theorem 5.1.]{Wan}):

\begin{fact}\label{number}
Let $a,m\in\mathbb{F}_q[t]$ be such that $deg(m)>0$ and $gcd(a,m)=1$. Let $N$ be a positive integer and let  
$$S_N(a,m)=\#\{f\in\mathbb{F}_q[t]~\text{monic irreducible}~|~f\equiv a ~(mod ~m),~ deg(f)=N\}.$$
Let $M=deg(m)$ and let $\Phi(m)$ denote the number of polynomials in $\mathbb{F}_q[t]$ relative prime to $m$ whose degree is smaller than M. Then we have the following inequality:
\begin{equation*}
|S_N(a,m)-\frac{q^N}{\Phi(m)N}|\leq \frac{1}{N}(M+1)q^{\frac{N}{2}}.
\end{equation*}
\end{fact}

As indicated in the Introduction, this fact is an extremely effective bound on the number of irreducible polynomials of a given degree in an arithmetic progression. A similar error term for prime numbers from an arithmetic progression in a given interval is not known.

\subsection{Gram-Schmidt orthogonalization}

\label{subsec:Gram-Schmidt}

We propose a version of the Gram-Schmidt orthogonalizition procedure
and prove a bound on the size of its output over $\mathbb{F}_q(t)$.
%\begin{definition}
%We call a bilinear form $h$ on the vector space $V$ regular if for every vector $x\in V$ there exists a vector $y\in V$ such that $h(x,y)\neq 0$. We call a subspace $U\leq V$ regular if $h$ restricted to $U$ is a regular bilinear form.
%\end{definition}
%\begin{remark}
%Note that $h$ is regular if and only if its Gram-matrix is non-singular (with respect to any basis).
%\end{remark}
\begin{lemma}\label{GS}
Let $(V,h)$ be an $n$-dimensional quadratic space over $\mathbb{F}_q(t)$. 
We assume that $h$ is given by its Gram matrix with respect to a 
basis $v_1,v_2,\dots, v_n$ whose entries are represented as fractions
of polynomials. Suppose that all the numerators occurring in the Gram matrix
have degree at most 
$\Delta$ while the degrees of the denominators are bounded by $\Delta'$. 
Then there is a deterministic
polynomial time algorithm which finds 
an orthogonal basis $w_1,\dots,w_n$ with respect to 
$h$ such that the maximum of the degrees of the numerators and 
the denominators of the $h(w_i,w_i)$ is $O(n(\Delta+\Delta'))$.
\end{lemma}
%\begin{remark}
%The extra condition concerning $u$ and $v$ may seem odd now. However, we will need this form later on.
%\end{remark}
\begin{proof}

We may assume that $h$ is regular. Indeed, we can compute the radical of $V$ by solving a system of linear equations and then continue in a direct complement of it. It is easy to select a basis for this direct complement as a subset of the original basis.  

We find an anisotropic vector $v_1'$ in the following way. 
If one of the $v_i$ is anisotropic then we choose $v_1':=v_i$. 
If all of them are isotropic then there must be an index $i$ such that 
$h(v_i,v_1)\neq 0$, otherwise $h$ would not be regular. 
Since $q$ is odd $v_1':=v_i+v_1$ will suffice.

Afterwards, we transform the basis $v_1,\dots,v_n$ into 
a basis $v_1',\dots,v_n'$ which has the property
that for every $k$, the subspace generated by $v_1',\dots, v_k'$ is regular. 
We start with $v_1'$ which is already anisotropic. 
Then we proceed inductively. We choose $v_{k+1}'$ in the following way. 
If some $j$ between $k+1$ and $n$ has the property that the subspace 
spanned by $v_1',\dots,v_k'$ and $v_j$ is regular then we choose 
$v_{k+1}':=v_j$ where $j$ is the smallest such index. Otherwise we claim that there exists an index $j$ 
between $k+1$ and $n$, that $v_{k+1}'=v_{k+1}+v_j$ is suitable. 
Note that if this is true then this can be checked in polynomial time. 
Indeed, the cost of the computation is dominated by that of computing
the determinants of the Gram matrices of the restriction of $h$ to the subspace 
spanned by $v_1',\dots, v_k'$ together with the candidate
$v_{k+1}'$. The number of these determinants is $O(n)$.  

Now we prove the claim. Let $U$ be now the subspace generated by 
$v_1',\dots, v_k'$ and let $\phi_U$ be the orthogonal projection 
onto the subspace $U$. Note that by our assumptions $U$ is a 
regular subspace and hence $V$ can be decomposed as the orthogonal
sum of the subspaces $U$ and $U^\perp$. Let $v^*=v-\phi_U(v)$, so $v^*$ 
is in the orthogonal complement of $U$. We have to prove that 
if neither $v_j$ is a suitable choice for $v_{k+1}'$ then there exists a 
$j$ such that $v_{k+1}+v_j$ is suitable. Note that if $v_{k+1}$ is 
not a suitable choice then the subspace generated by $U$ and $v_{k+1}^*$ 
is not regular (they generate the same subspace as $U$ and 
$v_{k+1}$) hence $v_{k+1}^*$ is isotropic because $U$ was regular. 
If for any $j$ between $k+1$ and $n$, the vector $v_j^*$ is anisotropic, 
we choose $v_{k+1}'=v_j^*$. Otherwise there must be a $j$ between $k+1$ 
and $n$ such that $h(v_{k+1}^*,v_j^*)\neq 0$ since $h$ is regular. 
This implies that $v_{k+1}^*+v_j^*$ is anisotropic since 
$h(v_{k+1}^*+v_j^*,v_{k+1}^*+v_j^*)=2h(v_{k+1}^*,v_j^*)\neq 0$. 
Observe that $v_{k+1}^*+v_j^*=(v_{k+1}+v_j)^*$ so $(v_{k+1}+v_j)^*$ is 
anisotropic. This implies that the subspace generated by 
$U$ and $v_{k+1}+v_j$ is regular.

Now we compute an orthogonal basis $w_1,\dots,w_n$ from the starting 
basis $v_1',\dots, v_n'$. We start with $w_1:=v_1'$. Let 
$w_k:=v_k'-u_k$ where $u_k$ is the unique vector from the subspace 
generated by $v_1',\dots,v_{k-1}'$  with the property 
that $h(u_i,v_j')=h(v_i',v_j')$ for every $j$ between 1 and $k$.
Uniqueness comes from the fact 
that $v_1',\dots, v_{k-1}'$ spans a regular subspace.

Finding $u_k$ is solving a system of $k$ linear equations with 
$k$ variables. Since the coefficient matrix of the system is 
non-singular because we chose $v_1',\dots, v_k'$ in this way, thus 
Cramer's rule applies. The same bounds on degrees apply to the 
Gram matrix obtained from the $v_i'$ as the original 
Gram matrix obtained from the $v_i$, since the transition matrix 
$T\in GL_n(\mathbb{F}_q)$. Hence Cramer's rule gives us the bounds 
on the $w_i$ as claimed.
\end{proof}

\subsection{Effective isotropy of binary and ternary quadratic forms
over $\mathbb{F}_q(t)$}

\label{subsec:effective-2-3}

We can efficiently diagonalize regular quadratic 
forms over $\mathbb{F}_q(t)$ using the version of the
Gram-Schmidt orthoginalization procedure discussed in Subsection~\ref{subsec:Gram-Schmidt}. Then a binary form can 
be made equivalent to
$b(x_1^2-ax_2^2)$ for some $a,b\in \mathbb{F}_q(t)$. The coefficient
$a$ is represented as the product of a scalar from $\mathbb{F}_q$ 
with the quotient of two monic polynomials. We can use the
Euclidean algorithm to make the quotient reduced. Then testing
whether $a$ is a square can be done in deterministic polynomial
time by computing the squarefree factorization of the two
monic polynomials and by computing the $\frac {q-1}{2}$th power of the scalar.
If $a$ is a square then a square root of it can be computed
by a randomized polynomial time method, the essential part of this
is computing a square root of the scalar constituent (\cite{Berlekamp},\cite{Shanks}).
Using this square root, 
linear substitutions ``standardizing'' 
hyperbolic forms (making them equivalent to $x_1^2-x_2^2$
or to $x_1x_2$, whichever is more desirable) can be computed
as discussed in Subsection~\ref{subsec:qf-gen}.

Non-trivial zeros of isotropic ternary quadratic forms can be computed 
in randomized polynomial time using the method of of Cremona and van Hoeij from \cite{Cremona1}. Through 
the connection with quaternion algebras described 
in Subsection~\ref{subsec:qf-gen}, the paper \cite{IKR} offers an
alternative approach. Here we cite the explicit bound on the size
of a solution from \cite[Section 1]{Cremona1}.

\begin{fact}\label{size}
Let $Q(x_1,x_2,x_3)=a_1x_1^2+a_2x_2^2+a_3x_3^2$ where 
$a_i\in\mathbb{F}_q[t]$. Then there is a randomized polynomial time 
algorithm 
%from \cite{Cremona1} 
which decides if $Q$ is isotropic and if it is, 
then computes a non-zero solution $(b_1,b_2,b_3)$ to $Q(x_1,x_2,x_3)=0$ 
with polynomials $b_1,b_2,b_3 \in\mathbb{F}_q[t]$ having 
the following degree bounds:
\begin{enumerate}[\rm(1)]

\item $deg(b_1)\leq deg(a_2a_3)/2,$
\item $deg(b_2)\leq deg(a_3a_1)/2,$
\item $deg(b_3)\leq deg(a_1a_2)/2.$
\end{enumerate}
\end{fact}

\section{Minimization and splitting}

In this section we describe the key ingredients needed 
for our algorithms for finding non-trivial zeros in 4 or 5 variables.
First we do some basic minimization to the quadratic form. 
Then we split the form $Q(x_1,\dots,x_n)$ (where $n=4$ or $n=5$) into two forms and 
show the existence of a certain value they both represent, assuming the 
original form is isotropic. The section is divided in two parts. 
The first deals with quadratic forms in 4 variables, the second with 
quadratic forms in 5 variables. 
%The third contains lemmas controlling 
%the sizes of solutions.

\subsection{The quaternary case}

We consider a quadratic form $Q(x_1,x_2,x_3,x_4)=a_1x_1^2+a_2x_2^2+a_3x_3^2+a_4x_4^2$. We assume that all the $a_i$ are in $\mathbb{F}_q[t]$ and are non-zero.

%We can also arrange that the $a_i$ are square-free. Indeed, if for a monic irreducible $f$, $f^2$ divides $a_i$ we replace $x_i$ by $x_i/f$.

We now give a simple algorithm which minimizes $Q$ in a certain way. We start with definitions:
\begin{definition}
We call a polynomial $h\in\mathbb{F}_q[t]$ {\em cube-free} if there do not exist any monic irreducible $f\in\mathbb{F}_q[t]$ such that $f^3$ divides $h$.
\end{definition}

Our goal is to replace $Q$ with another quadratic form $Q'$ which is isotropic if and only if $Q$ was isotropic and which has the property that from a non-trivial 
zero of $Q'$ a non-trivial zero  of $Q$ can be retrieved in polynomial time. For instance if we apply a linear change of variables to $Q$ (i.e., we replace $Q$ with an explicitly equivalent form), then this will be the case. However, we may further relax the notion of equivalence by allowing to multiply the quadratic form with a non-zero element from $\mathbb{F}_q(t)$. 
 
\begin{definition}
Let $Q$ and $Q'$ be diagonal quadratic forms in $n$ variables. We call $Q$ and $Q'$ projectively equivalent if $Q'$ can be obtained from $Q$ 
using the following two operations:
\begin{enumerate}[\rm(1)]
\item multiplication of $Q$ by a non-zero $g\in\mathbb{F}_q(t)$
\item linear change of variables
\end{enumerate}

We call these two operations {\em projective substitutions}. 
\end{definition}

%\begin{remark}
%This is not the usual definition of equivalence of quadratic forms. However in this sense if one finds an isotropic vector for one then one finds an isotropic vector for the other one in polynomial time.
%\end{remark}

\begin{definition}\label{minimized}
We call a diagonal quaternary quadratic form $$Q(x_1,x_2,x_3,x_4)=a_1x_1^2+a_2x_2^2+a_3x_3^2+a_4x_4^2$$ 
minimized if it satisfies the following four properties:
\begin{enumerate}[\rm(1)]
\item All the $a_i$ are square-free,
\item The determinant of $Q$ is cube-free,
\item If a monic irreducible $f$ does not divide $a_i$ and $a_j$ but divides the other two, then $-a_ia_j$ is a square modulo $f$,
\item The number of square leading coefficients among the $a_i$ is at least the number of non-square leading coefficients among the $a_i$.
\end{enumerate}
\end{definition}

\begin{remark}
By Lemma \ref{Lemma1} and Lemma \ref{Lemma1''}, a minimized quadratic form is locally isotropic at any finite prime.
\end{remark}

\begin{lemma}\label{minimization1}
There is a randomized algorithm running in polynomial time which either shows that $Q$ is anisotropic at a finite prime or returns the following data:
\begin{enumerate}[\rm(1)]
\item a minimized diagonal quadratic form $Q'$ which is projectively equivalent to $Q$,
\item a projective substitution which turns $Q$ into $Q'$.
\end{enumerate}
\end{lemma}

\begin{proof}
We factor each $a_i$. If for a monic irreducible polynomial $f$, $f^{2k}$ (where $k\geq 1$ ) divides $a_i$ then we substitute $x_i\leftarrow \frac{x_i}{f^k}$. By iterating this process through the list of primes dividing the $a_i$ we obtain a new equivalent diagonal quadratic form where all the coefficients are square-free polynomials.

Let $f$ be a monic irreducible polynomial in $\mathbb{F}_q[t]$ dividing the determinant of $Q$. If every $a_i$ is divisible by $f$ then we divide $Q$ by $f$. Now let us assume that $a_1$ is the only coefficient not divisible by $f$. Then we make the following substitution: $x_1\leftarrow fx_1$. This new form is still diagonal, and every coefficient is divisible by $f$. Moreover, $f^2$ divides exactly one of the coefficients. Divide the form by $f$. Then the multiplicity of $f$ in the determinant of the new form is exactly 1. If we do this for all monic irreducibles $f$, whose third power divides the determinant of $Q$, we obtain a new form whose determinant is cube-free.

Let us assume that each $a_i$ is square-free and that there exists a monic irreducible $f$ which divides exactly two of the $a_i$. We may assume that $f$ divides $a_1$ and $a_2$ but does not divide the other two coefficients. If $-a_3a_4$ is a square modulo $f$ we do nothing. If not, we do a change of variables $x_1\leftarrow x_1/f, x_2\leftarrow x_2/f$. If $-\frac{a_1}{f}\frac{a_2}{f}$ is not a square modulo $f$ then we can conclude that $Q$ is anisotropic in the $f$-adic completion by Lemma \ref{Lemma1''}. Otherwise we continue with the equivalent quadratic form $Q'(x_1,x_2,x_3,x_4)=\frac{a_1}{f}x_1^2+\frac{a_2}{f}x_2^2+fa_3x_3^2+fa_4x_4^2$. This is locally isotropic at $f$ due to Lemma \ref{Lemma1''}.

If the third condition is not satisfied then we multiply the quadratic form by a non-square element from $\mathbb{F}_q$.

Now we consider the running time of the algorithm. First we need to factor the determinant. There are factorisation algorithms which are randomized and run in polynomial time (\cite{Berlekamp}, \cite{CZ}). We might need a non-square element from $\mathbb{F}_q$. Such an element can be found by a randomized algorithm which runs in polynomial time. The rest of the algorithm runs in deterministic polynomial time (see Remark 1).
\end{proof}

%We would like to solve the following equation:

%\begin{equation*}
%a_1x_1^2+a_2x_2^2+a_3x_3^2+a_4x_4^2=0
%\end{equation*}

%Due to the minimization step described in the previous algorithm we may assume that $a_1a_2a_3a_4$ is cube-free and that it is isotropic at every finite prime. Our goal is to split equation (1) into two quadratic forms and find a common value they both represent.

The next lemma is the key observation for our main algorithm. 

\begin{lemma}\label{main1}
Assume that $a_1x_1^2+a_2x_2^2+a_3x_3^2+a_4x_4^2$ is an isotropic minimized quadratic form with the property that $a_ix_i^2+a_jx_j^2$ is anisotropic for every $i\neq j$. Let $D=a_1a_2a_3a_4$. Then there exists a permutation $\sigma\in S_4$, an $\epsilon\in \{0,1\}$ and a residue class $b$ modulo $D$ such that for every monic irreducible $a\in\mathbb{F}_q[t]$ satisfying $a\equiv b ~(mod ~D)$ and $deg(a)\equiv \epsilon~(mod ~2)$, the following equations are both solvable:

\begin{equation}
a_{\sigma(1)}x_{\sigma(1)}^2+a_{\sigma(2)}x_{\sigma(2)}^2=f_1\dots f_kg_1\dots g_la
\end{equation}
\begin{equation}
-a_{\sigma(3)}x_{\sigma(3)}^2-a_{\sigma(4)}x_{\sigma(4)}^2=f_1\dots f_kg_1\dots g_la
\end{equation}

Here $f_1,\dots,f_k$ are the monic irreducible polynomials dividing both $a_{\sigma(1)}$ and $a_{\sigma(2)}$. Also $g_1,\dots, g_l$ are the monic irreducibles dividing both $a_{\sigma(3)}$ and $a_{\sigma(4)}$. In addition, $b,\sigma$ and $\epsilon$ can be found by a randomized polynomial time algorithm.
\end{lemma}
\begin{remark}
The meaning of this lemma is that if we split the original quaternary form in an appropriate way into two binary quadratic forms then we can find this type of common value they both represent.
\end{remark}
\begin{proof}

First we show that with an arbitrary splitting into equations (2) and (3) we can guarantee local solvability (of equations (2) and (3)) everywhere by choosing $a$ in a suitable way except at infinity and at $a$. Then we choose $\sigma$ and $\epsilon$ in a way that local solvability is satisfied at infinity as well. Finally, Fact \ref{Product formula} shows local solvability everywhere.

For the first part we assume that $\sigma$ is the identity as this simplifies notations. 

Since $a_1x_1^2+a_2x_2^2$ or $a_3x_3^2+a_4x_4^2$ are anisotropic over $\mathbb{F}_q[t]$ the question whether equation (2) or (3) is solvable is equivalent to the following quadratic forms being isotropic over $\mathbb{F}_q(t)$:
\begin{equation}
a_1x_1^2+a_2x_2^2-f_1\dots f_kg_1\dots g_laz^2
\end{equation}
\begin{equation}
-a_3x_3^2-a_4x_4^2-f_1\dots f_kg_1\dots g_laz^2
\end{equation}

Due to the local-global principle (Theorem \ref{HM}) the quadratic forms (4) and (5) are isotropic over $\mathbb{F}_q(t)$ if they are isotropic locally everywhere. Hence equations (2) and (3) are solvable if and only if they are solvable locally everywhere.
\smallskip

Now we go through the set of primes excluding $a$ and infinity. We check local solvability at every one of them. We have 4 subcases for equation (2): the primes $f_i$; the primes $g_j$; primes dividing exactly one of $a_1$ and $a_2$; remaining primes. The list is similar for equation (3). First we show that (2) is solvable at all these primes.
\smallskip

\noindent{\textit{Solvability at the $f_i$}}

\smallskip
Equation (2) is solvable at any $f_i$ since we can divide by $f_i$ and obtain a quadratic form whose determinant is not divisible by $f_i$. By Fact \ref{not divide} this is solvable at $f_i$.
\smallskip

\noindent{\textit{Solvability at a prime $g$ which divides exactly one of $a_1$ and $a_2$}}

\smallskip
We may assume that $g$ divides $a_1$. Due to Lemma \ref{Lemma1} equation (2) is solvable in the $g$-adic completion if $a_2f_1\dots f_kg_1\dots g_la$ is a square modulo $g$ (meaning in the finite field $\mathbb{F}_q[t]/(g)$). Since $(\frac{a_2f_1\dots f_kg_1\dots g_l}{g})$ is fixed this gives the condition on $a$ that $(\frac{a}{g})=(\frac{a_2f_1\dots f_kg_1\dots g_l}{g})$. This can be thought of as a congruence condition on $a$ modulo $g$ (this gives a condition whether $a$ should be a square element modulo $g$ or not). Due to the Chinese Remainder Theorem these congruence conditions on $a$ can be satisfied simultaneously. This implies that $a$ has to be in one of certain residue classes modulo the product of these primes. We choose $a$ to be in one of these residue classes.
%owever before we proceed there is something we need to check. If $g$ divides exactly $a_1$ and $a_3$ we get that $-a_2a$ has to be square modulo $g$ and that $a_4a$ has to be a square modulo $g$. This can only be satisfied if $-a_2a_4$ is a square modulo $g$. Note that is satisfied since we started with a quadratic form that is minimized.

%However note that if $a_1$ and $a_3$ are divisible by a prime $g$ then it can happen that the congruence conditions on $a$ are contradictory (meaning it has to be a square and non-square at the same time). In this case look at the original equation (1) modulo $g$. Since $-a_2a_4$ will not be a square modulo $g$ this shows that the original equation (1) is not solvable modulo $g$ hence is not isotropic over $\mathbb{F}_q(t)$.
\smallskip
\noindent{\textit{Solvability at the $g_i$}}

Now consider equation (2) modulo the $g_i$. Note that due to minimization neither $a_1$ nor $a_2$ are divisible by the $g_i$. Hence equation (2) has a solution in the $g_i$-adic completion if and only if $-a_1a_2$ is a square modulo $g_i$. This is satisfied since we have a minimized quadratic form. 

\smallskip
\noindent{\textit{Solvability at the remaining primes}}
\smallskip

Solvability at these primes is satisfied by Fact \ref{not divide}.
\smallskip

%Note that solvability at all places is automatically satisfied (for every irreducible $a$) except at those monic irreducible polynomials which divide exactly one of $a_1$ and $a_2$. Hence the discussion of local solvability of (3) at these primes is the same.
Note that solvability of (2) holds independently of the choice of $a$ except for primes dividing exactly one of $a_1$ and $a_2$. Thus, in the analogous case of the solvability of (3) we have only to consider the case of primes which divide exactly one of $a_3$ and $a_4$. These impose congruence conditions again on $a$. A problem can occur if these congruence conditions are contradictory. We show that this cannot happen. Assume that a monic irreducible polynomial $g$ divides one of $a_1,a_2$ and one of $a_3,a_4$, say $a_1$ and $a_3$. By the previous discussion we have that in this case $-a_2af_1\dots f_kg_1\dots g_l$ should be a square modulo $g$ and that $a_4af_1\dots f_kg_1\dots g_l$ should be a square modulo $g$. These can always be satisfied by choosing $a$ to be in a suitable residue class modulo $g$ except if $-a_2a_4$ is not a square modulo $g$. However, this cannot happen since our form was minimized.
\medskip

Now we have proven that for any splitting, equations (2) and (3) are solvable locally everywhere for suitable primes $a$ except maybe at $a$ or at infinity.  We now choose $\sigma$ and the parity of the degree of $a$ in a way that both (2) and (3) are solvable at infinity. Then, by Fact \ref{Product formula}, (2) and (3) will be solvable at $a$ as well.

\medskip

First assume that all $a_i$ have odd degrees. Then we can pick $\sigma$ arbitrarily and we choose $a$ in a way that $f_1\dots f_kg_1\dots g_la$ has odd degree. Then both equations are solvable in $\mathbb{F}_q((\frac{1}{t}))$ by Lemma \ref{Lemma1'}, (1). 

\smallskip
Next assume that one coefficient is of even degree and all the others are of odd degree. Pick $\sigma$ in a way that $a_{\sigma(1)}$ is of even degree and the leading coefficient of $a_{\sigma(2)}$ is a square in $\mathbb{F}_q$. This can be achieved since we have a minimized quadratic form. Choose $a$ in a way that $f_1\dots f_kg_1\dots g_la$ has odd degree. Then equation (3) is solvable in $\mathbb{F}_q((\frac{1}{t}))$ due to the same reason as before. Equation (2) is also solvable due to Lemma \ref{Lemma1'}, (2).

\smallskip
Now assume that there are two odd degree coefficients and two even degree ones among the $a_i$. We have that at least two of the $a_i$ have a leading coefficient which is a square due to the fact that the form is minimized. We choose $\sigma$ in such a way that in equation (2) and (3) one coefficient is of odd degree and the other is of even degree. Assume $a_{\sigma(1)}$ and $-a_{\sigma(3)}$ are of odd degree. Let the leading coefficient of $a_i$ be $c_i$. If $c_{\sigma(1)}$ and $-c_{\sigma(3)}$ are both squares then we pick $a$ in a way that $f_1\dots f_kg_1\dots g_la$ has odd degree. If $c_{\sigma(2)}$ and $-c_{\sigma(4)}$ are both squares we pick $a$ in such a way that $f_1\dots f_kg_1\dots g_la$ has even degree. It may occur that $c_{\sigma(1)},c_{\sigma(2)},-c_{\sigma(3)},-c_{\sigma(4)}$ are all squares. In this case there is no degree constraint on $a$. In these two cases both equations are solvable at infinity by Lemma \ref{Lemma1'}. The only problem occurs if $c_{\sigma(1)}$ and $-c_{\sigma(3)}$ are not both squares and the same holds for $c_{\sigma(2)}$ and $-c_{\sigma(4)}$.

We distinguish two cases depending on whether $q\equiv 1~ (mod~ 4)$ or $q\equiv 3~ (mod~ 4)$. First suppose that $q\equiv 1~ (mod~ 4)$. In this case -1 is square element in $\F_q$. If neither $c_{\sigma(1)}$ nor $c_{\sigma(3)}$ is a square in $\F_q$ then $c_{\sigma(2)}$ and $c_{\sigma(4)}$ must be both squares. Therefore $-c_{\sigma(4)}$ is a square since -1 is a square and we have a contradiction because we assumed that one of $c_{\sigma(2)}$ and $-c_{\sigma(4)}$ is not a square. If neither $c_{\sigma(2)}$ nor $c_{\sigma(4)}$ is a square in $\F_q$ then $c_{\sigma(1)}$ and $c_{\sigma(3)}$ must be both squares which is again, a contradiction. The only problem occurs if exactly one of $c_{\sigma(1)}$ and $c_{\sigma(3)}$ is a square and the same is true for $c_{\sigma(2)}$ and $c_{\sigma(4)}$. However, in this case, the form $a_1x_1^2+a_2x_2^2+a_3x_3^2+a_4x_4^2$ is anisotropic by Remark \ref{inf}. 

Suppose that $q\equiv 3~ (mod~ 4)$. Note that in this case -1 is not a square in $\F_q$. If $c_{\sigma(1)}$ and $-c_{\sigma(3)}$ are non-squares then we have that $c_{\sigma(3)}$ is a square since -1 is not a square. Then let $\sigma'=\sigma\circ (13)$ (i.e., swap $a_{\sigma(1)}$ with $a_{\sigma(3)}$). Now $c_{\sigma'(1)}$ is a square and so is $-c_{\sigma'(3)}$, hence again we choose $a$ in a way that $f_1\dots f_kg_1\dots g_la$ has odd degree and equation (2) and (3) are solvable at infinity due to Lemma \ref{Lemma1'}. If $c_{\sigma(2)}$ and $-c_{\sigma(4)}$ are non-squares then the situation is essentially the same (let $\sigma'=\sigma\circ (24)$ and choose $a$ in a way that $f_1\dots f_kg_1\dots g_la$ has even degree). If exactly one of $c_{\sigma(1)}$ and $-c_{\sigma(3)}$ is a square and the same holds for $c_{\sigma(2)}$ and $-c_{\sigma(4)}$ then the form $a_1x_1^2+a_2x_2^2+a_3x_3^2+a_4x_4^2$ is anisotropic by Remark \ref{inf}. Indeed, $c_{\sigma(1)}$ and $c_{\sigma(3)}$ are either both squares or both non-squares and the same holds for $c_{\sigma(2)}$ and $c_{\sigma(4)}$.

%(note that if the leading coefficients of $a_{\sigma(1)}$ and $-a_{\sigma(3)}$ are both non-squares then the leading coefficients of $a_{\sigma(2)}$ and $-a_{\sigma(4)}$ are both squares because the form is minimized). 
%However, Lemma \ref{Lemma1'''} implies that in this case equation (1) is not solvable in $\mathbb{F}_q((\frac{1}{t}))$, so this cannot happen (note that $-1$ is a square if and only if $q\equiv 1 ~(mod~ 4)$ and see the remark after Lemma \ref{Lemma1'''}).

The cases where there is 1 odd degree one or no odd degree ones amongst the $a_i$ are esentially the same when there are three odd degree ones, or all are of odd degree. 

This shows that choosing $\sigma$ in this way equations (2) and (3) are solvable locally everywhere, except maybe at $a$, hence are solvable over $\mathbb{F}_q(t)$ as well by Fact \ref{Product formula}.

We conclude by verifying that $b, \sigma$ and $\epsilon$ can be found by a polynomial time algorithm. The computation of a residue class $b$ involves finding non-square elements in finite fields and Chinese remaindering. Both can be accomplished in polynomial time, the first using randomization. Choosing $\sigma$ and $\epsilon$ can be achieved in constant time by looking at the parity of the degrees of the $a_i$.
\end{proof}
%\begin{remark}
%Observe that solvability at infinity is analogous to solvability at a finite prime. Indeed, the number of odd degrees polynomials amongst the $a_i$ is similar to the valuation of the determinant at a finite prime. Also the leading coefficient of $a_i$ is the analogue of $a_i$ modulo a finite prime.
%\end{remark}

\begin{remark}
As seen in the proof there is not just one residue class $b$ modulo $D$ that would satisfy the necessary conditions. Assume that $D$ is divisible by $k$ different monic irreducible polynomials. Then $q^{deg(D)}/3^k$ is a lower bound on the number of appropriate residue classes. Indeed, since modulo each prime half of the non-zero residue classes are squares. However, we will not use this fact later on.
\end{remark}

\medskip

\subsection{The 5-variable case}

We consider a quadratic form $Q(x_1,x_2,x_3,x_4,x_5)=a_1x_1^2+a_2x_2^2+a_3x_3^2+a_4x_4^2+a_5x_5^2$, where the $a_i\in\mathbb{F}_q[t]$ are non-zero polynomials. 

%We again may assume that all the $a_i\in\mathbb{F}_q[t]$ are square-free. First we do a similar minimization as in Lemma \ref{minimization1}.
\begin{lemma}\label{minimization2}
There exists a randomized polynomial time algorithm that returns a projectively equivalent diagonal quadratic form $Q'$ whose coefficients are square-free polynomials and whose determinant is cube-free, and a projective substitution which transforms $Q$ into $Q'$. 
\end{lemma}
\begin{proof}

Making the  coefficients of $Q'$ square-free is done in a similar fashion as in Lemma \ref{minimization1}. If every coefficient is divisible by a monic irreducible $f$ we divide $Q$ by $f$. If at most 2 coefficients are not divisible by $f$ we do the same trick as in Lemma \ref{minimization1}. To implement this for every irreducible polynomial $f$, we need to factor the determinant. This can be achieved in polynomial time by a randomized algorithm \cite{Berlekamp}. All the other steps run in deterministic polynomial time.
\end{proof}

\medskip

Now we prove a lemma similar to Lemma \ref{main1}.
\begin{lemma}\label{main2}
Let $Q(x_1,x_2,x_3,x_4,x_5)=a_1x_1^2+a_2x_2^2+a_3x_3^2+a_4x_4^2+a_5x_5^2$, where $D=a_1a_2a_3a_4a_5$ is cube-free and all the $a_i$ are square-free polynomials from $\mathbb{F}_q[t]$. Suppose, $a_ix_i^2+a_jx_j^2+a_kx_k^2$ is anisotropic for every $1\leq i<j<k\leq 5$. Then there exists a permutation $\sigma\in S_5$, an $\epsilon\in \{0,1\}$ and a residue class $b$ modulo $D$ such that for every monic irreducible $a\in \mathbb{F}_q[t]$ satisfying $a\equiv b~ (mod ~ D)$ and $deg(a)\equiv \epsilon ~(mod~ 2)$ the following equations are both solvable:

\begin{equation}
a_{\sigma(1)}x_{\sigma(1)}^2+a_{\sigma(2)}x_{\sigma(2)}^2=f_1\dots f_ka
\end{equation}
\begin{equation}
-a_{\sigma(3)}x_{\sigma(3)}^2-a_{\sigma(4)}x_{\sigma(4)}^2-a_{\sigma(5)}x_{\sigma(5)}^2=f_1\dots f_ka
\end{equation}

Here $f_1,\dots,f_k$ are the monic irreducible polynomials dividing both $a_{\sigma(1)}$ and $a_{\sigma(2)}$. In addition, $b,\sigma$ and $\epsilon$ can be found by a randomized polynomial time algorithm.
\end{lemma}

\begin{remark}
Assuming that $a_ix_i^2+a_jx_j^2+a_kx_k^2$ is anisotropic for every $i,j,k$ allows us to consider the solvability of equations (6) and (7) as finding nontrivial zeros of the quadratic forms $a_{\sigma(1)}x_{\sigma(1)}^2+a_{\sigma(2)}x_{\sigma(2)}^2-f_1\dots f_kaz^2$ and $-a_{\sigma(3)}x_{\sigma(3)}^2-a_{\sigma(4)}x_{\sigma(4)}^2-a_{\sigma(5)}x_{\sigma(5)}^2-f_1\dots f_kaz^2$ hence we can use our lemmas and theorems from the previous sections.
\end{remark}
\begin{proof}

First we show that for any $\sigma\in S_5$ equation (6) is solvable for suitable $a$ at any prime except maybe at infinity and at $a$. Also if $a$ is suitably chosen then equation (7) is solvable everywhere except maybe at infinity. In order to simplify notation we can assume that $\sigma$ is the identity.
\smallskip

First consider equation (6). It is solvable at any of the $f_i$ since $a_1$ and $a_2$ are square-free (Lemma \ref{Lemma1}). It is solvable at any prime not dividing $a_1a_2f_1\dots f_ka$ by Fact \ref{not divide}. Let $g$ be a prime that divides $a_1$ but not $a_2$. In order to ensure that (6) is solvable in the $g$-adic completion $-a_2af_1\dots f_k$ has to be a square modulo $g$. This imposes a congruence condition on $a$. The situation is the same when looking at a prime dividing $a_2$ but not $a_1$.
\smallskip

Now consider equation (7). Again if a prime does not divide any of the coefficients then the equation is locally solvable at that prime. The equation is solvable at every $f_i$ (using (1) of Lemma \ref{Lemma1} with $z=0$) since none of the $f_i$ divide $a_3,a_4,a_5$. Similarly it is also solvable at $a$ because we choose $a$ to differ from the primes occurring in $a_3a_4a_5$. If a prime $g$ divides exactly one of $a_3,a_4,a_5$ then similarly the equation is locally solvable at that prime. Finally consider the case where a prime $h$ divides exactly two out of $a_3,a_4,a_5$ (say $a_3$ and $a_4$). This gives a congruence condition on $a$. Specifically, $-af_1\dots f_ka_5$ has to be a square modulo $h$. Note that since for every prime $f$, $f^3$ does not divide the determinant of the original quadratic form, the congruence conditions on $a$ coming from equations (6) and (7) cannot be contradictory.
\smallskip

Now we choose $\sigma$ and $\epsilon$ in a way that both (6) and (7) become solvable at infinity at the cost of possibly restricting the parity of the degree of $a$. Then by Fact \ref{Product formula} equation (6) will become solvable at $a$ as well. Finally by the local-global principle (Theorem \ref{HM}) both equations are solvable over $\mathbb{F}_q[t]$.

First if all $a_i$ have odd degree then $\sigma$ can be chosen arbitrarily and we choose $a$ in a way that $f_1\dots f_ka$ has odd degree. This way both equations are solvable at infinity by Lemma \ref{Lemma1'}, (1).
\smallskip

Now consider the case where one coefficient has even degree and the others are of odd degree. Then we choose $\sigma$ in a way that $a_{\sigma(3)}$ has even degree and the others are of odd degree. We choose $a$ in a way that $f_1\dots f_ka$ has an odd degree. Due to Lemma \ref{Lemma1'} both equations are solvable at infinity (with $x_{\sigma(3)}=0$).
\smallskip

Finally assume that there are two $a_i$-s with even degree. We choose $\sigma$ in a way that $a_{\sigma(1)}$ and $a_{\sigma(2)}$ are of even degree. We choose $a$ in such a way that $f_1\dots f_ka$ has even degree. Now equations (6) and (7) are solvable at infinity. The remaining cases are essentially the same, we systematically swap "odd" and "even" in the preceding arguments. 

Note that $b, \sigma$ and $\epsilon$ can be found in polynomial time using randomization by the same reasoning as described at the end of the proof of Lemma \ref{main1}. 
\end{proof}

\newpage
\section{The main algorithms}

In this section we describe two algorithms. One for solving a quadratic equation in 4 variables and one for 5 variables. The algorithms are similar, however, the second uses the first algorithm. The idea of the algorithms is the following. Split the original equation into two and find a common value they both represent and then solve the two equations.

The input of the first algorithm is a diagonal quadratic form $$Q(x_1,x_2,x_3,x_4)=a_1x_1^2+a_2x_2^2+a_3x_3^2+a_4x_4^2,$$ where all $a_i$ are in $\mathbb{F}_q[t]$.

\hrulefill
\begin{algorithm}[Quaternary case]\label{alg1}
\medskip
\begin{enumerate}[\rm(1)]

\item Minimize $Q$ using the algorithm from Lemma \ref{minimization1}. Minimization either yields that $Q$ is anisotropic (then stop) or returns a new projectively equivalent quadratic form $Q'(x_1,x_2,x_3,x_4)=b_1x_1^2+b_2x_2^2+b_3x_3^2+b_4x_4^2$ which is minimized. If $b_1,b_2,b_3,b_4\in\F_q$ then return a non-trivial zero of $Q'$ using the algorithm of $\cite{W}$.

\item Check solvability at infinity (Remark \ref{iso} and Lemma \ref{Lemma1'''}). Check if $b_ix_i^2+b_jx_j^2$ is isotropic for every pair $i\neq j$. If it is for a pair $(i,j)$ then return a solution.

\item Split the quadratic form into equations (2) and (3) (i.e., find a suitable permutation $\sigma\in S_4$) as discussed in Lemma \ref{main1}.

\item List the congruence conditions on $a$ as described in Lemma \ref{main1} and solve this system of linear congruences. Obtain a residue class $b$ modulo $b_1b_2b_3b_4$ as a result.

\item Let $d$ be the degree of $b_1b_2b_3b_4$ and let $N=4d$ or $N=4d+1$ depending on the degree parity $\epsilon$ we need by Lemma \ref{main1}. Pick a random polynomial $f$ of degree $N$ of the residue class $b$ modulo $b_1b_2b_3b_4$ and check whether it is irreducible. If $f$ is irreducible, then proceed. If not, then repeat this step.

\item Solve equations (2) and (3) using the method of \cite{Cremona1}.

\item By subtracting equation (3) from equation (2) find a non-trivial zero of $Q'$.

\item Return a non-trivial zero of $Q$ using the inverse substitutions of the substitutions obtained by the algorithm from Lemma \ref{minimization1}.
\end{enumerate}
\hrulefill
\end{algorithm}

\medskip

The input of the second algorithm is a quadratic form $$Q(x_1,x_2,x_3,x_4,x_5)=a_1x_1^2+a_2x_2^2+a_3x_3^2+a_4x_4^2+a_5x_5^2,$$ where all $a_i$ are non-zero polynomials in $\mathbb{F}_q[t]$.
\medskip

\hrulefill

\begin{algorithm}
\medskip
\begin{enumerate}[\rm(1)]

\item Minimize $Q$ using the algorithm from Lemma \ref{minimization2}. Minimization returns a new projectively equivalent diagonal quadratic form $Q'(x_1,x_2,x_3,x_4,x_5)=b_1x_1^2+b_2x_2^2+b_3x_3^2+b_4x_4^2+b_5x_5^2$ whose determinant is cube-free and whose coefficients are square-free. If $b_1,b_2,b_3,b_4,b_5\in\F_q$ then return a non-trivial zero of $Q'$ using the algorithm of $\cite{W}$.

\item Split the quadratic form into equations (6) and (7) as discussed in the proof of Lemma \ref{main2}. Check if the quadratic forms on the left-hand side of equations of (6) and (7) are isotropic or not. If one of them is, then return a non-trivial solution. Use the algorithm from \cite{Cremona1}.

\item List the congruence conditions on $a$ as described in Lemma \ref{main2} and solve this system of linear congruences. Obtain a residue class $b$ modulo $b_1b_2b_3b_4b_5$ as a result.

\item Let $d$ be the degree of $b_1b_2b_3b_4b_5$ and let $N=4d$ or $N=4d+1$ according to degree parity $\epsilon$ we need by Lemma \ref{main2}. Pick a random polynomial $f$ of degree $N$ of the residue class $b$ modulo $b_1b_2b_3b_4b_5$ and check whether it is irreducible. If $f$ is irreducible,  then proceed. If not, then repeat this step.

\item Solve equations (6) and (7) using the method of \cite{Cremona1} and Algorithm \ref{alg1}.

\item By subtracting equation (3) from equation (2) find a non-trivial zero of $Q'$.

\item Return a non-trivial zero of $Q$ using the inverse substitutions of the substitutions obtained by the algorithm from Lemma \ref{minimization2}.

\end{enumerate}
\hrulefill
\end{algorithm}

\begin{theorem}\label{main}
Algorithm 1 and Algorithm 2 are randomized algorithms of Las Vegas type which run in polynomial time in the size of the quadratic form (the largest degree of the coefficients) and in $\log~ q$. Let $D$ be the determinant of the quadratic form. Let $d=deg(D)$. Algorithm 1 either detects that the form is anisotropic or returns a solution of size $O(d)$, that is an array of 4 polynomials of degree $O(d)$. Algorithm 2 always returns a solution of size $O(d)$, that is an array of 5 polynomials of degree $O(d)$.
\end{theorem}

\begin{proof}

The correctness of the algorithms follows from Lemmas \ref{main1} and \ref{main2}. We start analyzing the running times of the algorithms. First we deal with Algorithm 1. We consider its running time step by step.
The first part of Step 1 runs in polynomial time (is however randomized) as proven in Lemma \ref{minimization1}. The second part of Step 1 is deterministic and runs in polynomial time (see \cite{W}). From now on we suppose that the determinant of the minimized form has degree at least 1. The first part of Step 2 can be executed in deterministic polynomial time (using Fact \ref{fact1} combined with Lemma \ref{Lemma1'} and \ref{Lemma1'''}). The second part is checking whether a polynomial is a square due to Fact \ref{fact1}. This can be done in polynomial time by computing the square-free factorization of the polynomial (\cite{Yun}) and checking whether the leading coefficient is a square or not (Remark \ref{square2}). Step 3 runs in deterministic polynomial time since we only need to check whether certain leading coefficients are squares in $\mathbb{F}_q$ or not. In Step 4 in order to obtain congruence conditions we may have to present a non-square element in a finite field (an extension of $\mathbb{F}_q$ which has degree smaller than the determinant of $Q'$). This can be done by a randomized algorithm which runs in polynomial time. Note that the probability that a non-zero random element in a finite field of odd characteristic is a square is 1/2. In the other part of Step 4 we have to solve a system of linear congruences. This can be done in deterministic polynomial time by Chinese remaindering.

Step 5 needs more explanation. After solving the linear congruences we obtain a residue class $b$ modulo $D$ (Lemma \ref{main1}). By Fact \ref{number} we have that (note that $d\geq 1$):
\begin{equation*}
\left|S_{N}(b,D)-\frac{q^N}{\Phi(D)N}\right|\leq \frac{1}{N}(d+1)q^{\frac{N}{2}}.
\end{equation*} 

We choose the degree of $a$ to be $N=4d$ or $N=4d+1$ depending on the parity we need for the degree of $a$ which is discussed in the proof of Lemma \ref{main1}. We give an estimate on the probability that a polynomial in this given residue class is irreducible. We have the following:
$$\frac{S_{N}(b,D)}{q^{N-d}}\geq \frac{q^N}{q^{N-d}\Phi(D)N}-\frac{(d+1)q^{\frac{N}{2}}}{Nq^{N-d}}\geq \frac{1}{N}-\frac{d+1}{Nq^{\frac{N}{2}-d}}\geq \frac{1}{N}-\frac{d+1}{Nq^d}\geq \frac{1}{3N}$$

%\leq \frac{d+1}{6d}\frac{1}{q^{2d}}<1/12d=\frac{1}{2N}.

Here we used the fact that $\frac{d+1}{q^{d}}\leq 2/3$ since $q\geq 3$ and the function $\frac{d+1}{q^{d}}$ is decreasing (as a function of $d$). We also used that $q^d\geq\Phi(D)$. 

We pick a uniform random monic element $a$ from the residue class $b$ modulo $D$. This can be done in the following way. We pick a random polynomial $r(t)\in\mathbb{F}_q[t]$ of degree $N-d$ whose leading coefficient is the inverse of the leading coefficient of $D$. We consider the polynomial $r':=rD+b$. Then $r'$ has degree $N$, is monic and is congruent to $b$ modulo $D$.

The probability that $a$ is irreducible is at least $1/3N$ by the previous calculation. Irreducibility can be checked in deterministic polynomial time \cite{Berlekamp}. This means that the probability that we do not obtain an irreducible polynomial after $3N$ tries is smaller than $1/2$. Hence this step runs in polynomial time (it is, however, randomized).

The last two steps use the algorithm from \cite{Cremona1}. This algorithm is randomized and runs in polynomial time.

The discussion for Algorithm 2 is similar.

Now we turn to the question of the size of solutions. First we consider Algorithm 1. The previous discussion shows that $N$ (the degree of $a$) can be chosen to be of size $O(d)$ . Finally when solving equations (2) and (3) we use the algorithm from \cite{Cremona1}. By Fact \ref{size} we obtain that the solution for (2) and (3) have size $O(d)$.
In the case of Algorithm 2 the same reasoning is valid, except that we have to use Algorithm 1 for solving (7).
\end{proof}

\begin{remark}
Due to Fact \ref{fact2} and Theorem \ref{HM} we have that every quadratic form in 5 or more variables is isotropic over $\mathbb{F}_q(t)$. Hence Algorithm naturally works for diagonal quadratic forms in more than 5 variables. Indeed, we set some variables to zero and use Algorithm 2.
\end{remark}
\begin{corollary}\label{not diagonal}
Assume that $Q$ is a regular quadratic form (not necessarily diagonal) in either 4 or 5 variables. Let $D$ be the determinant of $Q$. Let $d_1$ be the largest degree of all numerators of entries of the Gram matrix of $Q$. Let $d_2$ be the largest degree of all denominators of entries of the Gram matrix of $Q$. Then there is randomized polynomial time algorithm which finds a non-trivial zero of $Q$ of size $O(d_1+d_2)$.
\end{corollary}
\begin{proof}
First we diagonalize $Q$ using Lemma \ref{GS}. As a result we obtain a quadratic form with determinant $D'$. The degree of the numerator and the denominator of $D'$ are both of size $O(d_1+d_2)$. By clearing the denominators we obtain a quadratic form $Q''$ with polynomial coefficients and determinant of degree $O(d_1+d_2)$. Using Algorithm 1 or 2 (depending on the dimension) we find an isotropic vector. By Theorem \ref{main} the size of the solution vector is $O(d_1+d_2)$.
\end{proof}
\begin{remark}
Corollary \ref{not diagonal} can be extended to higher dimensions as well. We diagonalize the quadratic form and then set all $x_i$ to zero except 5. Then apply Algorithm 2. Due to diagonalization the size of the solution in this case is $O(n(d_1+d_2))$.
\end{remark}

%Algorithm 1 is essentially a randomized reduction from solving a homogeneous quadratic equation in 4 variables to solving several homogeneous equations in 3 variables. It is well-known (for instance see in \cite{Castel}) that solving a quadratic equation in 4 variables is at least as hard as solving one in three variables. Indeed, since a homogeneous quaternary equation of square determinant always reduces to a homogeneous ternary equation and every homogeneous ternary equation arises this way.

%This motivates the following natural a question: is there a deterministic reduction from the quaternary case to the ternary one? Or we can consider the following weaker question. Is there a reduction where we only need to factor polynomials over finite fields, otherwise the reduction is deterministic? We think the answer to the weaker question is yes, however the algorithm is probably more complicated than Algorithm 1 of this paper.

\section{Equivalence of quadratic forms}

In this section we use the algorithms from the previous sections to 
compute the following: the Witt decomposition of a quadratic form, 
a maximal totally isotropic subspace and the transition matrix for
two equivalent quadratic forms. We use a presentation in the context
of quadratic spaces. We assume that a quadratic space is input by
the Gram matrix with respect to a basis.

\begin{theorem}\label{Wittalg}
Let $(V,h)$ be a regular quadratic space, 
$V=\mathbb{F}_q(t)^n$. 
There exists a randomized polynomial time algorithm which 
finds a Witt decomposition of $(V,h)$.
\end{theorem}
\begin{proof}
First we find an orthogonal basis using Lemma \ref{GS}. 
This basis can be used to decompose the space into the orthogonal sum 
of subspaces of dimension 5 and possibly one quadratic form of dimension 
at most 4, each with an already computed
orthogonal basis.
In every 5 dimensional subspace we find an isotropic vector using Algorithm 2. Then we find a hyperbolic plane in each of these subspaces. The subspace generated by this isotropic vector and one of the basis elements from the orthogonal basis of the subspace will be suitable because otherwise $h$ would not be regular restricted to this subspace. We compute its orthogonal complement inside this 5 dimensional subspace. These are all of dimension 3. We find
an orthogonal basis in
each of these 3 dimensional subspaces using Lemma \ref{GS}. For their direct sum we
again have an orthogonal basis and we iterate the process (we again group by 5 and find hyperbolic planes). We have that $V$ is the orthogonal sum of hyperbolic planes and a subspace of dimension at most 4. Using Algorithm 1 
for the quaternary case, the algorithm from \cite{Cremona1} for the ternary
case, and the method of Subsection~\ref{subsec:effective-2-3}
if the dimension is 2, we either conclude that it is anisotropic 
or find a decomposition into hyperbolic planes and anisotropic part.

Now consider the running time of the algorithm. Assume that $h$ was given by a Gram matrix where the maximum degree of the numerators is $\Delta$ and the maximum degree of the denominators is $\Delta'$. Diagonalization is done in polynomial time via Lemma \ref{GS}. Also, it produces a diagonal Gram matrix where every numerator and denominator has degree at most $n(\Delta+\Delta')$. Afterwards we only diagonalize in dimension at most 5. Hence in each step the degrees only grow by a constant factor by Corollary \ref{not diagonal}. The number of iterations is
$O(\log n)$ so the algorithm will run in polynomial time (it is however randomized since Algorithm 1 and 2 are randomized).
\end{proof}
\begin{corollary}
Let $h$ be a regular bilinear form on the vector space $V=\mathbb{F}_q(t)^n$. Then, there exists a randomized polynomial time algorithm which finds a maximal totally isotropic subspace for $h$.
\end{corollary}
\begin{proof}
We compute the Witt decomposition of $h$ using Theorem \ref{Wittalg}. Then we take an isotropic vector from each hyperbolic plane. They generate a maximal totally isotropic subspace \cite[Chapter I, Corollary 4.4.]{Lam}.
\end{proof}

Here we only considered regular bilinear forms. Now we deal with the case where $h$ is not regular.

\begin{corollary}\label{not regular}
Let $(V,h)$ be a quadratic space. There exists a randomized polynomial 
time algorithm which finds a Witt decomposition of $h$.
\end{corollary}
\begin{proof}
The radical of $V$ can be computed by solving a system of linear equations. Then $h$ restricted to
a direct complement of the radical is regular, thus Theorem \ref{Wittalg} applies.
\end{proof}

We conclude the section by proposing an algorithm for explicit equivalence of quadratic forms. For simplicity we restrict our attention to regular bilinear forms.
\begin{theorem}\label{equivalence}
Let $(V_1,h_1)$ and $(V_2,h_2)$ be regular quadratic forms over 
$\mathbb{F}_q(t)$. 
%with Gram-matrices $G_1$ and $G_2$. 
Then there exists a randomized polynomial time algorithm which 
decides whether they are isometric,
and, in case they are, computes an isometry between them.
\end{theorem}
\begin{proof}
The quadratic spaces $(V_1,h_1)$ and $(V_2,h_2)$ are equivalent 
if and only if the orthogonal sum of $(V_1,h_1)$ and $(V_2,-h_2)$ 
can be decomposed into the orthogonal sum of hyperbolic 
planes (\cite[Chapter I, Section 4]{Lam}). 
Hence the question of deciding isometry can be solved using Theorem \ref{Wittalg}. 
We turn our attention to the second part of the theorem, to computing an isometry.	

First we consider the case of quadratic spaces whose Witt decomposition 
consist only of the orthogonal sum of hyperbolic planes
(i.e., hyperbolic spaces). As shown in Subsection~\ref{subsec:qf-gen},
we can transform each of the corresponding binary forms into the standard
diagonal form, $x_1^2-x_2^2$. This results in new bases for the two spaces
in which $h_1$ and $h_2$ have block diagonal matrices with $2\times 2$
diagonal blocks
$$\begin{pmatrix}
		1 & 0 \\
		0 & -1
\end{pmatrix}.$$
%Hence again by a simple change of variables 
%we arrive at a diagonal form with 1 and -1 in the diagonal.
The linear extension of an approriate bijection 
between these bases is an isometry. We can 
efficiently compute the matrix of this map
in terms of the original bases.

Let us assume now that $(V_1,h_1)$
and $(V_2,h_2)$ are isometric anisotropic quadratic spaces. 
Isometry implies that $(V_1\oplus V_2,h_1\oplus -h_2)$ 
is the orthogonal sum of hyperbolic planes. We find
a basis of $V_1\oplus V_2$ in which the Gram matrix of 
$h_1\oplus -h_2$ is of a block diagonal form like above.
Then the substitution described in Subsection~\ref{subsec:qf-gen}
for equivalence of the two standard binary hyperbolic forms
$x_1^2-x_2^2$ and $x_1x_2$ can be used to construct a new basis
$b_1,b_2,\dots , b_{2n}$ in which
the Gram matrix becomes block diagonal with blocks 
$$\begin{pmatrix}
	0  & \frac{1}{2} \\
	\frac{1}{2}  & 0
\end{pmatrix}.$$
(Here $n$ 
is the common dimension of $V_1$ and $V_2$.)
Every $b_i$ can be uniquely written 
in the form $b_i=u_i+v_i$ where $u_i\in V_1$ and $v_i\in V_2$. 
These can be found by orthogonal projection.
We claim that the vectors $u_1,u_3\ldots, u_{2n-1}$ are linerly
independent. 
To see this, assume that
$$\lambda_1u_1+\lambda_3 u_3+\ldots+ \lambda_{2n-1}u_{2n-1}=0$$
for some $\lambda_1,\ldots,\lambda_{2n-1}$ not all zero.
Then the vector 
$b=\lambda_1b_1+\lambda_3 b_3+\ldots+ \lambda_{2n-1}b_{2n-1}$
is non-zero as the $b_i$ are linearly independent. The orthogonal
projection of $b$ to $V_1$ is zero, whence $b$ is a non-zero vector
from $V_2$. The vector $b$, as a member of
the totally isotropic subspace spanned
by $b_1,b_3,\ldots,b_{2n-1}$,
must be isotropic. This however contradicts to the anisotropy of $(V_2,-h_2)$.
Therefore $u_1,u_3,\ldots,u_{2n-1}$ is a basis of $V_1$. By symmetry,
$v_1,v_3,\ldots,v_{2n-1}$ is a basis of $V_2$.
 Now we prove that the Gram matrix of the quadratic form $h_1$ 
in the basis $u_1,u_3\dots, u_{2n-1}$ is the same as the 
Gram matrix of $h_2$ in the basis $v_1,v_3\dots,v_{2n-1}$. 
Observe that since the Gram matrix of 
$h_1\oplus -h_2$ had zeros in the diagonal 
$h_1(u_i,u_i)=h_2(v_i,v_i)$. 
Since we chose only the odd indices 
(i.e there are no two indices which differ by 1) 
we also have that $h_1(u_i,u_j)=h_2(v_i,v_j)$. 
Thus the linar extension of the map $u_i\rightarrow v_i$
($i=1,3,\ldots,2n-1$) is an isometry between $V_1$ and $V_2$.
One only has to compute the matrix of this map in terms
of the original bases for $V_1$ and $V_2$. 

In order to find isometries of  possibly isotropic 
quadratic spaces we first compute their Witt decomposition. 
Then by \cite[Chapter I, Section 4]{Lam} we know that they 
are isometric if and only if their 
hyperbolic and anisotropic parts are isometric respectively. 
An isometry can be found by taking the direct sum of a pair of
isometries between the respective parts. Again, one can finish with
computing the matrix of this direct sum map in terms
of the original bases for $V_1$ and $V_2$. 
\end{proof}

\begin{remark}
Theorem \ref{equivalence} can be extended to 
degenerate quadratic spaces by using Corollary \ref{not regular}.
Also, the proof actually shows existence of a reduction from computing 
isometries to three instances of computing Witt decompositions of quadratic 
spaces over an arbitrary field of characteristic different from 2.
\end{remark}

\section{An application}

Besides equivalence of quadratic forms,
the explicit isomorphism problem with full $2\times 2$
matrix algebras for global function fields provides further
motivation for solving homogeneous quadratic equations in 4 and 5 
variables.
We now describe the explicit isomorphism problem in more detail. 
Let $\K$ be a field, $\A$ an associative algebra over $\K$. Suppose that $\A$ is isomorphic to the full matrix algebra $M_n(\K)$. The task is to construct explicitly an isomorphism $\A\rightarrow M_n(K)$. Or, equivalently, give an irreducible $\A$-module.

Recall, that for an algebra $\A$ over a field $\K$ and for a $\K$-basis   $a_1, \ldots ,a_m$ of $\A$ over $\K$, the products $a_ia_j$ can be expressed as linear combinations of the $a_i$:
$$ a_ia_j=\gamma _{ij1}a_1+\gamma _{ij2}a_2+\cdots +\gamma _{ijm}a_m. $$

The elements $\gamma _{ijk}\in \K$ are called structure constants. We consider $\A$ to be given by a collection of
structure constants.

The case when $\K=\mathbb{F}_q(t)$ is considered in \cite{IKR}, where a randomized polynomial time algorithm is proposed for computing an explicit isomorphism. However, when $\K$ is a finite extension of $\mathbb{F}_q(t)$, the same problem remained open. The only known algorithms for this task run in time exponential in the degree of the extension and the degree of the discriminant of the extension. The first interesting case is when $n=2$ and $\K$ is a quadratic extension of $\mathbb{F}_q(t)$. Here we solve this problem using Algorithms 1 and 2. The method is a straightforward analogue of the algorithm from \cite{K}.

Let $\K$ be a quadratic extension of $\mathbb{F}_q(t)$. Let $\A$, an algebra isomorphic to $M_2(\K)$, be given by structure constants. First we find a subalgebra in $\A$ which is quaternion algebra over $\mathbb{F}_q(t)$. This is done in two steps. We begin with finding an element $u$ in $\A$ such that $u^2\in\mathbb{F}_q(t)$ and $u$ is not in the center of $\A$. Then we find an element $v$ such that $uv+vu=0$ and $v^2\in\mathbb{F}_q(t)$. Finally, the $\mathbb{F}_q(t)$-vector space generated by $1,u,v,uv$ yields the desired subalgebra. In the first step of this algorithm we make use of Algorithm 2. In the second part we make use of Algorithm 1.

Recall, that we denoted by $H_\F(\alpha,\beta)$ the quaternion algebra over the field $\F$ (if $\chara (\F)\neq 2$) with parameters $\alpha,\beta\in\F^*$. 
%(i.e. it has a quaternion basis $1,u,v,uv$ such that $u^2=\alpha$ ,$v^2=\beta$ and $uv=-vu$).

Let $K=\mathbb{F}_q(t)(\sqrt{d})$, where $d$ is a square-free polynomial in $\mathbb{F}_q[t]$.

\begin{proposition}\label{square}
Let $\A\cong M_2(\K)$ be given by structure constants. Then there exists a randomized polynomial time algorithm which finds a non-central element $l$, such that $l^2\in\mathbb{F}_q(t)$.
\end{proposition}
\begin{proof}
 First we construct a quaternion basis $1,w,w',ww'$ of $\A$ in the following way. We find a non-central element $w$ such that $w^2\in\K$ by completing the square and then find an element $w'$ such that $ww'+w'w=0$. This can be found by solving a system of linear equations. Such a $w'$ exists by the following reasoning. The map $\sigma: s\mapsto ws+sw$ is $\K$-linear and has a non-trivial kernel since its image is contained in the centralizer of $w$ (which is not $\A$ since $w$ was non-central). Then $1,w,w',ww'$ will be a quaternion basis. Details can be found in \cite{Ronyai}.

 We have the following:
\begin{equation*}
w^2=r_1+t_1\sqrt{d}, ~w'^2=r_2+t_2\sqrt{d}.
\end{equation*}

Here $r_1,r_2,t_1,t_2\in\F_q(t)$. In order to ensure that the square of $l$ is in $\K$ it has to be in the $\K$-subspace generated by $w$, $w'$ and $ww'$ (\cite[Section 1.1.]{Vigneras}). In other words the element $l$ is of the form $l=(s_1+s_2\sqrt{d})w+(s_3+s_4\sqrt{d})w'+(s_5+s_6\sqrt{d})ww'$, where $s_1,\dots,s_6\in\mathbb{F}_q(t)$. The condition $l^2\in\mathbb{F}_q(t)$ is equivalent to the following: 
\begin{equation*}
((s_1+s_2\sqrt{d})w+(s_3+s_4\sqrt{d})w'+(s_5+s_6\sqrt{d})ww')^2\in \mathbb{F}_q(t).
\end{equation*}

If we expand this we obtain:
\begin{eqnarray*}
((s_1+s_2\sqrt{d})w+(s_3+s_4\sqrt{d})w'+(s_5+s_6\sqrt{d})ww')^2=\\
(s_1^2+ds_2^2+2s_1s_2\sqrt{d})(r_1+t_1\sqrt{d})+(s_3^2+ds_4^2+2s_3s_4\sqrt{d})(r_2+t_2\sqrt{d})-\\
(s_5^2+ds_6^2+2s_5s_6\sqrt{d})(r_1+t_1\sqrt{d})(r_2+t_2\sqrt{d}).
\end{eqnarray*}

In order for $l$ to be in $\mathbb{F}_q(t)$ the coefficient of $\sqrt{d}$ has to be zero:
\begin{eqnarray*}
t_1s_1^2+t_1ds_2^2+2r_1s_1s_2+t_2s_3^2+t_2ds_4^2+2r_2s_3s_4-(r_1t_2+t_1r_2)s_5^2-\\
(r_1t_2+t_1r_2)ds_6^2-2(r_1r_2+t_1t_2d)s_5s_6=0.
\end{eqnarray*}

The previous equation can be solved by Algorithm 2. Note that a quadratic form in 6 variables over $\mathbb{F}_q(t)$ is always isotropic.
\end{proof}

Now we turn to the second step:

\begin{proposition}\label{anticommute}
Let $B=H_{\K}(a,b+c\sqrt{d})$ given by: $u^2=a, v^2=b+c\sqrt{d}$, where $a,b,c\in \mathbb{F}_q(t), c\neq 0$. Then one can find a $v'$ (if it exists) in randomized polynomial time such that $uv'+v'u=0$ and $v'^2\in\mathbb{F}_q(t)$.
\end{proposition}
\begin{proof}
Since $v'$ anticommutes with $u$ (i.e. $uv'+v'u=0$) it must be a $\K$-linear combination of $v$ and $uv$. Indeed, the map $\sigma: B\rightarrow B$ defined by $s\mapsto us+su$ is linear whose image has dimension at least 2 over $K$ ($2u$ and $2a$ are in the image). Therefore its kernel has dimension at most 2 and actually exactly 2 since $v$ and $uv$ are in the kernel. 
 
This means that we have to search for $s_1,s_2,s_3,s_4\in\mathbb{F}_q(t)$ such that:
\begin{equation*}
((s_1+s_2\sqrt{d})v+(s_3+s_4\sqrt{d})uv)^2\in\mathbb{F}_q(t).
\end{equation*}

Expanding this expression we obtain the following:
\begin{eqnarray*}
((s_1+s_2\sqrt{d})v+(s_3+s_4\sqrt{d})uv)^2=\\
(s_1^2+s_2^2d+2s_1s_2\sqrt{d})(b+c\sqrt{d})-(s_3^2+s_4^2d+2s_3s_4\sqrt{d})a(b+c\sqrt{d}).
\end{eqnarray*}

In order for this to be in $\mathbb{F}_q(t)$, the coefficient of $\sqrt{d}$ has to be zero. So we obtain the following equation:

\begin{equation}
c(s_1^2+s_2^2d)+2bs_1s_2-ac(s_3^2+s_4^2d)-2abs_3s_4=0.
\end{equation}

Thus we have proven that finding a $v'$ satisfying the conditions of the proposition is equivalent to solving equation (8). We either detect that equation (8) is not solvable or return a solution using Algorithm 1.
\end{proof}

\begin{remark}
Actually a little bit of calculation shows that one only needs the algorithm from \cite{Cremona1} to solve equation (8) (\cite{K}). 
\end{remark}

Finally we state these results in one proposition:
\begin{proposition}\label{subalgebra}
Let $\A\cong M_2(\K)$ be given by structure constants. Then one can find either a four dimensional subalgebra over $\mathbb{F}_q(t)$ which is a quaternion 
algebra, or a zero divisor, by a randomized algorithm which runs in polynomial time.
\end{proposition}
\begin{proof}

First we find a non-central element $l$ such that $l^2\in\mathbb{F}_q(t)$. If $l^2=r^2$, where $r\in\F_q(t)$ then we return the zero divisor $l-r$ which is non-zero since $l$ is non-central. Otherwise, when $l^2$ is not a square in $\F_q(t)$, one finds an element $l'$ such that $ll'+l'l=0$ and $l'^2\in\mathbb{F}_q(t)$. These can be done using Proposition \ref{square} and \ref{anticommute}. If $l'^2=0$ we again have a zero divisor. If not, then the $\mathbb{F}_q(t)$-space generated by $1,l,l',ll'$ is a quaternion algebra over $\mathbb{F}_q(t)$. The only thing we need to show is that for any $l$ such an $l'$ exists.

%There exists a subalgebra $\A_0$ in $\A$ which is isomorphic to $M_2(\mathbb{F}_q(t))$. In this subalgebra there is an element $l'$ for which $l$ and $l'$ have the same minimal polynomial over $K$. This means there exists an $m\in A$ such that $l=m^{-1}l'm$. Hence $m^{-1}A_0m$ will contain $l$.
There exists a subalgebra $\A_0$ in $\A$ which is isomorphic to $M_2(\mathbb{F}_q(t))$. In this subalgebra there is an element $l_0$ for which $l$ and $l_0$ have the same minimal polynomial over $K$. This means that there exists an $m\in \A$ such that $l=m^{-1}l_0m$ (\cite[Theorem 1.2.1.]{Vigneras}). There exists a non-zero $l_0'\in\A_0$ such that $l_0l_0'+l_0'l_0=0$ (the existence of such an $l_0'$ was already proven at the beginning of the proof of Proposition \ref{square}). Let $l'=m^{-1}l_0'm$. We have that $l'^2=m^{-1}l_0'mm^{-1}l_0m=m^{-1}l_0^2m=l_0^2$, hence $l'^2\in\mathbb{F}_q(t)$. Since conjugation by $m$ is an automorphism we have that $ll'+l'l=m^{-1}(l_0l_0'+l_0'l_0)m=m^{-1}0m=0$. Thus we have proven the existence of a suitable element $l'$.
\end{proof}

Now we show how to apply this result to find a zero divisor in $\A$:
\begin{proposition}\label{application}
Let $\A\cong M_2(\K)$ be given by structure constants. Then there exists a randomized polynomial time algorithm which finds a zero divisor in $\A$.
\end{proposition}
\begin{proof}
We invoke the algorithm from Proposition \ref{subalgebra}. If it returns a zero divisor, then we are done. If not, then we have quaternion subalgebra $H$ over $\mathbb{F}_q(t)$. If $H$ is isomorphic to $M_2(\mathbb{F}_q(t))$, then one can find a zero divisor in it by using the algorithm form \cite{Cremona1} (or \cite{IKR}). If not, then there exists an element $s\in H$ such that $s^2=d$. Indeed, $H$ is split by $\K$ and therefore contains $\K$ as a subfield \cite[Theorem 1.2.8]{Vigneras}. Let $1,u,v,uv$ be a quaternion basis of $H$ with $u^2=a,v^2=b$. Every non-central element whose square is in $\F_q(t)$ is an $\mathbb{F}_q(t)$-linear combination of $u$, $v$ and $uv$. Hence finding an element $s$ such that $s^2=d$ is equivalent to solving the following equation:
\begin{equation}
ax_1^2+bx_2^2-abx_3^2=d.
\end{equation}

Since $H$ is a division algebra, the quadratic form $ax_1^2+bx_2^2-abx_3^2$ is anisotropic. Thus solving equation (9) is equivalent to finding an isotropic vector for the quadratic form $ax_1^2+bx_2^2-abx_3^2-dx_4^2$. One can find such a vector using Algorithm 1. We have found an element $s$ in $H$ such that $s^2=d$. Since $H$ is a central simple algebra over $\mathbb{F}_q(t)$ and $d$ is not a square in $\mathbb{F}_q(t)$, the element $s$ is not in the center of $\A$. Hence $s-\sqrt{d}$ is a zero divisor in $\A$.
\end{proof}
\begin{remark}
Let $\F$ be any field whose characteristic is different from 2 and let $\K$ be a quadratic extension of $\F$. The above described procedure reduces the question of finding a non-trivial zero of a ternary quadratic form over $\K$ to finding non-trivial zeros of quadratic forms of 4 or more variables over $\F$. 
\end{remark}
\smallskip

We also give another application of Algorithm 2 concerning quaternion algebras.
\begin{definition}
Let $\F$ be field such that $\chara~\F\neq 2$. We call two quaternion algebras $A_1=H_\F(a_1,b_1), A_2=H_\F(a_2,b_2)$ linked if there exist an element $\alpha\in \F$ such that $A_1=H_\F(\alpha,x)$ and $A_2=H_\F(\alpha,y)$.
\end{definition}

It is known (\cite[Chapter III, Theorem 4.8.]{Lam}) that over $\mathbb{F}_q(t)$ any two quaternion algebras are linked. We now propose an algorithm which finds such a presentation.
\begin{proposition}
Let $A_1=H_{\mathbb{F}_q(t)}(a_1,b_1), A_2=H_{\mathbb{F}_q(t)}(a_2,b_2)$, with $a_1,a_2,b_1,b_2\in\mathbb{F}_q(t)^*$. Then, there exists a randomized polynomial time algorithm which finds $\alpha\in\mathbb{F}_q(t)$ such that $A_1=H_{\mathbb{F}_q(t)}(\alpha,x)$ and $A_1=H_{\mathbb{F}_q(t)}(\alpha,y)$.
\end{proposition}
\begin{proof}
Consider the quadratic form $a_1x_1^2+b_1x_2^2-a_1b_1x_3^2-a_2x_4^2-b_2x_5^2+a_2b_2x_6^2$. Find an isotropic vector for this quadratic form using Algorithm 2. Let the solution vector be $(y_1,\dots,y_6)$. Then let $\alpha=a_1y_1^2+b_1y_2^2-a_1b_1y_3^2=a_2y_4^2+b_2y_5^2-a_2b_2y_6^2$. If $\alpha=0$ then $A_1\cong A_2\cong M_2(\mathbb{F}_q(t))$, hence such a presentation can be found using the algorithm from \cite{Cremona1}. If $\alpha\neq 0$ then let $1,u_1,v_1,u_1v_1$ be the quaternion basis of $A_1$. Then the task of finding a suitable presentation reduces to finding an element which anticommutes with $y_1u_1+y_2v_1+y_3(u_1v_1)$. This can be done in polynomial time. The same goes for $A_2$.
\end{proof}
\begin{remark}
This problem can also be thought of as calculating a common splitting field of two quaternion algebras.
\end{remark}
\paragraph*{Acknowledgement.}
Research supported by the Hungarian National Research, Development and Innovation Office - NKFIH, Grant K115288. The authors are grateful to an anonymous referee for helpful remarks and suggestions. 

\end{document}